\documentclass[
a4paper,
11pt,
twoside,
]{article}
\usepackage[utf8]{inputenc}
\usepackage[T1]{fontenc}
\usepackage[sc]{mathpazo}
\usepackage[english]{babel}
\usepackage{geometry}
\usepackage{amsmath,amsfonts,amssymb,amsthm}
\usepackage[final]{hyperref}
\usepackage{graphicx}
\usepackage{mathtools}
\usepackage[dvipsnames]{xcolor}
\usepackage{pdflscape}
\usepackage{etoolbox}
\usepackage{fancyhdr}
\usepackage{lastpage}
\usepackage{ifdraft}
\usepackage{hyphenat}
\usepackage{enumitem}
\usepackage{siunitx}
\usepackage{booktabs}
\usepackage[mathlines,pagewise]{lineno}

\newcommand*\patchAmsMathEnvironmentForLineno[1]{%
  \expandafter\let\csname old#1\expandafter\endcsname\csname #1\endcsname
  \expandafter\let\csname oldend#1\expandafter\endcsname\csname end#1\endcsname
  \renewenvironment{#1}%
  {\linenomath\csname old#1\endcsname}%
  {\csname oldend#1\endcsname\endlinenomath}}%
\newcommand*\patchBothAmsMathEnvironmentsForLineno[1]{%
  \patchAmsMathEnvironmentForLineno{#1}%
  \patchAmsMathEnvironmentForLineno{#1*}}%
\patchBothAmsMathEnvironmentsForLineno{equation}%
\patchBothAmsMathEnvironmentsForLineno{align}%
\patchBothAmsMathEnvironmentsForLineno{flalign}%
\patchBothAmsMathEnvironmentsForLineno{alignat}%
\patchBothAmsMathEnvironmentsForLineno{gather}%
\patchBothAmsMathEnvironmentsForLineno{multline}%

\usepackage{csquotes}

\newcommand{\dd}{\mathop{}\!\mathrm{d}}
\newcommand{\evaluatedin}{\raise-.5ex\hbox{\ensuremath{\vert}}}

\numberwithin{equation}{section}
\numberwithin{table}{section}
\numberwithin{figure}{section}

\usepackage[capitalize,nameinlink,sort]{cleveref}[0.19]
\crefname{section}{section}{sections}
\crefname{subsection}{section}{sections}
\Crefname{figure}{Figure}{Figures}
\crefformat{equation}{\textup{#2(#1)#3}}
\crefrangeformat{equation}{\textup{#3(#1)#4--#5(#2)#6}}
\crefmultiformat{equation}{\textup{#2(#1)#3}}{ and \textup{#2(#1)#3}}
{, \textup{#2(#1)#3}}{, and \textup{#2(#1)#3}}
\crefrangemultiformat{equation}{\textup{#3(#1)#4--#5(#2)#6}}%
{ and \textup{#3(#1)#4--#5(#2)#6}}{, \textup{#3(#1)#4--#5(#2)#6}}{, and \textup{#3(#1)#4--#5(#2)#6}}
\Crefformat{equation}{#2Equation~\textup{(#1)}#3}
\Crefrangeformat{equation}{Equations~\textup{#3(#1)#4--#5(#2)#6}}
\Crefmultiformat{equation}{Equations~\textup{#2(#1)#3}}{ and \textup{#2(#1)#3}}
{, \textup{#2(#1)#3}}{, and \textup{#2(#1)#3}}
\Crefrangemultiformat{equation}{Equations~\textup{#3(#1)#4--#5(#2)#6}}%
{ and \textup{#3(#1)#4--#5(#2)#6}}{, \textup{#3(#1)#4--#5(#2)#6}}{, and \textup{#3(#1)#4--#5(#2)#6}}
\crefdefaultlabelformat{#2\textup{#1}#3}

\crefname{chapter}{chapter}{chapters}
\crefname{appendix}{appendix}{appendices}
\crefname{subappendix}{section}{sections}
\Crefname{subappendix}{Section}{Sections}
\crefname{page}{page}{pages}

\usepackage[%
norefs,%
nocites%
]{refcheck}
\makeatletter
\newcommand{\refcheckize}[1]{%
  \expandafter\let\csname @@\string#1\endcsname#1%
  \expandafter\DeclareRobustCommand\csname relax\string#1\endcsname[1]{%
    \csname @@\string#1\endcsname{##1}\@for\@temp:=##1\do{\wrtusdrf{\@temp}\wrtusdrf{{\@temp}}}}%
  \expandafter\let\expandafter#1\csname relax\string#1\endcsname
}
\newcommand{\refcheckizetwo}[1]{%
  \expandafter\let\csname @@\string#1\endcsname#1%
  \expandafter\DeclareRobustCommand\csname relax\string#1\endcsname[2]{%
    \csname @@\string#1\endcsname{##1}{##2}\wrtusdrf{##1}\wrtusdrf{{##1}}\wrtusdrf{##2}\wrtusdrf{{##2}}}%
  \expandafter\let\expandafter#1\csname relax\string#1\endcsname
}
\makeatother
\refcheckize{\cref}
\refcheckize{\Cref}
\refcheckizetwo{\crefrange}
\refcheckizetwo{\Crefrange}
\refcheckize{\labelcref}
\refcheckize{\lcnamecref}
\refcheckize{\lcnamecrefs}
\refcheckize{\namecref}
\refcheckize{\namecrefs}
\refcheckize{\cpageref}
\refcheckize{\pagerefstar}

\DeclareMathOperator*{\esssup}{ess\,sup}
\DeclareMathOperator{\spanop}{span}
\DeclareMathOperator{\Lip}{Lip}

\renewcommand{\restriction}{\raise-.5ex\hbox{\ensuremath{\upharpoonright}}}

\usepackage{mathtools}
\DeclarePairedDelimiter{\abs}{\lvert}{\rvert}
\DeclarePairedDelimiter{\norm}{\lVert}{\rVert}

\usepackage{enumitem}

\newlist{H-hyp}{enumerate}{2}
\setlist[H-hyp,1]{label=(H\arabic*)}
\setlist[H-hyp,2]{label=(H\arabic{H-hypi}.\arabic*)}
\crefname{H-hyp}{hypothesis}{hypotheses}
\Crefname{H-hyp}{Hypothesis}{Hypotheses}
\crefalias{H-hypi}{H-hyp}
\crefalias{H-hypii}{H-hyp}
\makeatletter
\let\cref@old@resetby\cref@resetby%
\def\cref@resetby#1#2{
  \let#2\relax%
  \ifnum\pdfstrcmp{#1}{H-hypii}=\z@
    \def#2{H-hypi}%
  \fi%
 \ifx#2\relax%
    \cref@old@resetby{#1}{#2}
 \fi}%
\makeatother

\theoremstyle{plain}
\newtheorem{theorem}{Theorem}[section]
\newtheorem{lemma}[theorem]{Lemma}
\newtheorem{proposition}[theorem]{Proposition}
\newtheorem{corollary}[theorem]{Corollary}

\fancypagestyle{mypagestyle}{%
  \fancyhf{}%
  \fancyhead[LO,RE]{\scriptsize Approximating evolution operators of linear delay equations}%
  \fancyhead[RO,LE]{\scriptsize Andò, Bosco, Breda, Liessi}%
  \fancyfoot[LE,RO]{\scriptsize \thepage\,/\,\ifdraft{LastPage}{\pageref{LastPage}}}%
}
\title{Approximating evolution operators of linear delay equations: a general framework for the convergence analysis}
\author{
Alessia And\`o$^{1}$, Giusy Bosco$^{2}$, Dimitri Breda$^{3}$ and Davide Liessi $^{4}$\\[.5em]
\small 
CDLab -- Computational Dynamics Laboratory\\[-.2em]
\small Department of Mathematics, Computer Science and Physics -- University of Udine\\[-.2em]
\small via delle scienze 206, 33100 Udine, Italy\\[.5em]
\small $^{1}$\texttt{alessia.ando@uniud.it}\\[-.2em]
\small $^{2}$\texttt{giusy.bosco1998@gmail.com}\\[-.2em]
\small $^{3}$\texttt{dimitri.breda@uniud.it}\\[-.2em]
\small $^{4}$\texttt{davide.liessi@uniud.it}
}
\date{\today}
\begin{document}
\clearpage
\maketitle
\thispagestyle{empty}
\begin{abstract}
We consider the problem of discretizing evolution operators of linear delay equations with the aim of approximating their spectra, which is useful in investigating the stability properties of (nonlinear) equations via the principle of linearized stability.
We develop a general convergence analysis based on a reformulation of the operators by means of a fixed-point equation, providing a list of hypotheses related to the regularization properties of the equation and the convergence of the chosen approximation techniques on suitable subspaces.
This framework unifies the proofs for some methods based on pseudospectral discretization, which we present here in this new form.
To exemplify the generality of the framework, we also apply it to a method of weighted residuals found in the literature, which was previously lacking a formal convergence analysis.
\end{abstract}



\section{Introduction}

Delayed terms frequently arise in mathematical modeling, as they allow to base the current evolution on the past history, enhancing the realism of the model.
Examples of fields where delays occur naturally are control theory in engineering, e.g.
\cite{%
Fridman2014,%
GuKharitonovChen2003,%
MichielsNiculescu2014,%
Stepan1989%
},
and population dynamics or epidemics in mathematical biology, e.g.
\cite{%
ArinoVanDenDriessche2006,%
BredaDiekmannDeGraafPuglieseVermiglio2012,%
Kuang1993,%
MacDonald1978,%
MetzDiekmann1986,%
Smith2011%
};
further applications are discussed in
\cite{%
Erneux2009,%
KolmanovskiiMyshkis1999%
}.

Delays may have several effects on the dynamics of a system: they typically have a destabilizing effect, but they may also contribute to stabilization; moreover, they can facilitate the emergence of periodic orbits.
In applications, there is thus a strong interest in studying the asymptotic stability of equilibria, periodic orbits and other kinds of invariant sets.

A key approach consists in resorting to the principle of linearized stability, which links the problem of interest to the stability properties of the systems obtained via linearization around the relevant solutions (or, more precisely, to the stability of the null solution of such linearized systems).
These properties can be derived from the evolution operators, i.e. the operators that advance the state of the system by a certain time span, which act on the infinite-dimensional state space.
In the case of equilibria, in principle any evolution operator gives the required information, while for periodic orbits the monodromy operators, which advance the state by one period, are considered \cite{DiekmannGettoGyllenberg2008,DiekmannVanGilsVerduynLunelWalther1995,HaleVerduynLunel1993,BredaLiessi2021}.
Evolution operators can also be used to investigate the presence of chaotic dynamics via the Lyapunov exponents \cite{BredaVanVleck2014,BredaLiessi2025}.

In previous works, some of the authors and colleagues have introduced methods for approximating the evolution operators of linear delay differential equations or retarded functional differential equations (RFDEs) \cite{BredaMasetVermiglio2012,BredaMasetVermiglio2015}, of linear renewal equations (REs) \cite{BredaLiessi2018} and of coupled REs and RFDEs \cite{BredaLiessi2020}, reducing them to finite dimension with pseudospectral discretization techniques, also in a piecewise fashion \cite{BredaLiessiVermiglio2022}.

Several other methods to approximate the evolution operators and their spectra are available in the literature
\cite{%
ItoKappel1991,%
LuzyaninaEngelborghsLustRoose1997,%
InspergerStepan2011,%
BorgioliHajduInspergerStepanMichiels2020,%
Bueler2007,%
ButcherBobrenkov2011,%
ButcherMaBuelerAverinaSzabo2004,%
LehotzkyInsperger2016,%
LehotzkyInspergerStepan2016,%
KhasawnehMann2011a,%
KhasawnehMann2011b,%
KhasawnehMann2013,%
Breda2006,%
EngelborghsLuzyaninaRoose2002%
}.
They provide strong experimental evidence for convergence, but most of them lack a formal proof.

The aim of this work is twofold.
On the one hand, inspired by \cite{BredaMasetVermiglio2012,BredaLiessi2018}, we provide a unified convergence proof, identifying and highlighting the role of the essential hypotheses related to the class of equation, the choice of numerical method and their interplay.
For instance, it turns out that a key point is the regularizing effect that the equation has on its solutions (see \cref{H-FsVm-to-Xhat,H-FsVp-to-Xhat}).
On the other hand, the proof contained in this work is in fact more general and opens the way to the rigorous justification of the other methods mentioned above: an example is shown in \cref{sec:weightedresiduals}.

\bigskip

We conclude this introduction presenting the two classes of equations we are considering, namely linear RFDEs and REs.

Let $d \in \mathbb{N}$ and $\tau \in \mathbb{R}$ both positive and consider a space $X$ of functions $[-\tau, 0]\to\mathbb{R}^{d}$ equipped with a norm denoted by $\norm{\cdot}_{X}$.
For $s \in \mathbb{R}$ and a function $x$ defined on $[s - \tau, + \infty)$ let
\begin{equation}\label{segment}
x_{t}(\theta) \coloneqq x(t + \theta), \qquad t \geq s, \; \theta \in [-\tau, 0].
\end{equation}

A linear RFDE is an equation of the form 
\begin{equation}\label{RFDE}
x'(t) = L(t)x_t,
\end{equation}
where $x'$ denotes the right-hand derivative of $x$, $\mathbb{R} \times X \ni (t, \phi) \mapsto L(t) \phi \in \mathbb{R}^{d}$ is a continuous function, linear in the second argument, and the state space is $X \coloneqq C([-\tau, 0], \mathbb{R}^{d})$ with $\norm{\cdot}_{X}$ the usual uniform norm.
The properties of $L$ imply that $L(t) \colon X \to \mathbb{R}^{d}$ is a linear bounded functional for all $t \in \mathbb{R}$ and $L(\cdot) \phi \colon \mathbb{R} \to \mathbb{R}^{d}$ is a continuous function for all $\phi \in X$.
A typical form of RFDE, which we take as prototype, is
\begin{equation}\label{RFDE-prototype}
x'(t) = A(t) x(t) + \sum_{k = 1}^{r} B_k(t) x(t-\tau_{k})
+ \sum_{k = 1}^{r} \int_{-\tau_{k}}^{-\tau_{k-1}} C_k(t, \theta) x(t + \theta) \dd\theta,
\end{equation}
with $\tau_{0} \coloneqq 0 < \tau_{1} < \dots < \tau_{r} \coloneqq \tau$, which corresponds to
\begin{equation}\label{L_RFDE}
L(t)\phi = A(t) \phi(0) + \sum_{k = 1}^{r} B_k(t) \phi(-\tau_{k})
+ \sum_{k = 1}^{r} \int_{-\tau_{k}}^{-\tau_{k-1}} C_k(t, \theta) \phi(\theta) \dd\theta.
\end{equation}

A linear RE is an equation of the form
\begin{equation}\label{RFE}
x(t) = \int_{- \tau}^{0} C(t, \theta) x_{t}(\theta) \dd \theta,
\end{equation}
where $C \colon \mathbb{R} \times [-\tau, 0] \rightarrow \mathbb{R}^{d \times d}$ is a measurable function and the state space is $X \coloneqq L^{1}([-\tau, 0], \mathbb{R}^{d})$ with $\norm{\cdot}_{X}$ the usual $L^{1}$ norm.

We will assume that the initial value problem (IVP) defined by the relevant equation has a unique solution, so that it allows to define a dynamical system and the family of evolution operators (see, e.g., \cite{HaleVerduynLunel1993,BredaLiessi2018}).

\bigskip

The remaining sections of this work will be organized as follows. In \cref{sec:reformulation}, we recall the definition of evolution operators and describe a reformulation which is convenient for separating the role of the class of equation from that of the right-hand side.
In \cref{sec:numerical} we describe the discretization framework in detail, further exploiting the proposed reformulation.
In \cref{sec:convergence} we present the fixed-point equation arising from the reformulation, as well as our general convergence analysis.
Then in \cref{sec:case} we describe some numerical methods for RFDEs and REs and verify that our general approach allows us to prove convergence in these cases. Concluding comments follow in \cref{sec:discussion}.

\section{Evolution operators and a convenient reformulation}
\label{sec:reformulation}


Consider an interval $J$ of real numbers and let
\begin{equation*}
\triangle \coloneqq \{(t, s) \in \mathbb{R}^{2} \mid s, t \in J \text{ and } s \leq t\}.
\end{equation*}
Let $U=\{U(t,s)\}_{(t,s)\in\triangle}$ be a \emph{(forward) evolutionary system}, or \emph{evolution family}, on $X$ \cite{ChiconeLatushkin1999,DiekmannVanGilsVerduynLunelWalther1995}, i.e., a family of linear and bounded operators $U(t, s) \colon X \to X$ with the properties that $U(s,s)=I_X$ (the identity of $X$) for each $s\in J$ and $U(t,r)U(r,s)=U(t,s)$ (the \emph{semigroup law}) for each $(t,r), (r,s) \in \triangle$.
The evolutionary system $U$ is called \emph{strongly continuous} if for each $x \in X$ the function $\triangle \ni (t, s) \mapsto U(t, s) x \in X$ is continuous.
For a comprehensive treatment, see \cite{DiekmannVanGilsVerduynLunelWalther1995,Pazy1983,ChiconeLatushkin1999}.
In this work we mainly exploit the properties of determinism represented by the semigroup law, consequence of the well-posedness of the corresponding IVPs.
In this respect, due to the presence of delay, the initial history has a fundamental role.

Let $s \in J$ and let $h \in \mathbb{R}$ be positive and such that $s+h\in J$, and define for brevity
\begin{equation}\label{T}
T \coloneqq U(s + h, s).
\end{equation}
We assume that $T$ can be reformulated as follows.

Consider the space $X^{+}$ of functions $[0, h]\to\mathbb{R}^{d}$ and the space $X^{\pm}$ of functions $[-\tau, h]\to\mathbb{R}^{d}$ equipped, respectively, with the norms $\norm{\cdot}_{X^{+}}$ and $\norm{\cdot}_{X^{\pm}}$: these functions and norms are of the same kind as those of $X$.%
\footnote{Ideally, the problem described by the evolution family $U$ concerns functions defined on $[-\tau,+\infty)$ having a certain regularity, and the spaces $X$, $X^+$ and $X^{\pm}$ are obtained by restricting those functions to the respective intervals.
Observe in particular that an element of one of those spaces can be transformed into an element of another via translation of the variable and restriction to appropriate intervals or prolongation by, e.g., constant value $0$.}
We assume that there exist linear operators $V \colon X \times X^{+} \to X^{\pm}$ and $\mathcal{F}_{s} \colon X^{\pm} \to X^{+}$ such that $V(\phi,z)_0=V(\phi,z)\restriction_{[-\tau,0]}=\phi$ for each $(\phi,z)\in X\times X^+$ and, for each $\phi\in X$,
\begin{equation}\label{T-as-V}
T \phi = V(\phi, z^{\ast})_{h},
\end{equation}
where $z^{\ast} \in X^{+}$ is the solution of the fixed point equation 
\begin{equation}\label{fixed_point}
z = \mathcal{F}_{s} V(\phi, z),
\end{equation}
which we assume to exist uniquely.
Observe that in~\cref{T-as-V} the subscript~$h$ is used according to~\cref{segment}.
The operator $V$ has the role of reconstructing the solution from the initial history and the information given by the right-hand side.
Such information is given by the operator $\mathcal{F}_s$ based on the solution, and is obtained by solving \cref{fixed_point}.

For convenience, we also define $V^{-} \colon X \to X^{\pm}$ and $V^{+} \colon X^{+} \to X^{\pm}$ be given, respectively, by $V^{-} \phi \coloneqq V(\phi, 0_{X^{+}})$ and $V^{+} z \coloneqq V(0_{X}, z)$, where $0_{Y}$ denotes the null element of a linear space $Y$.
Observe that
\begin{equation}\label{decomp_V}
V(\phi, z) = V^{-} \phi + V^{+} z
\end{equation}
and that $(V^+ z)\restriction_{[-\tau,0]} = 0_X$.

\section{Numerical reduction to finite dimension}
\label{sec:numerical}

The problem is reduced to finite dimension by using two possibily different methods in $X$ and in $X^{+}$.
Let $M,N\in\mathbb{N}$ be the corresponding discretization indices; together they determine the dimension of the reduced problem.

For the reduction of $X$ we consider a finite dimensional space $X_M$ and a linear operator $R_M\colon \widetilde{X} \to X_M$, called \emph{restriction operator}, where $\widetilde{X}$ is a subspace of $X$ on which the chosen reduction method is defined.%
\footnote{This is motivated by the case of \cite{BredaLiessi2018}, in which the evolution operators of renewal equations are discretized with pseudospectral techniques involving interpolation: pointwise evaluation of functions does not make sense in the $L^1$ state space, but is well-defined in the subspace of continuous functions.}
For convenience of notation, assume that the dimension of $X_M$ is $d(M+1)$, although it need not be.
In the following we identify $X_M$ and $\mathbb{R}^{d(M+1)}$ via a certain basis, and given $\Phi\in X_M$ we denote its components as $(\Phi_0^T,\dots,\Phi_M^T)^T$, where $\Phi_i\in\mathbb{R}^d$ for each $i$.
We also consider a linear operator $P_M\colon X_M\to X$, called \emph{prolongation operator}, such that
\begin{equation*}
R_{M} P_{M} = I_{X_{M}}, \qquad
P_{M} R_{M} = \mathcal{L}_{M},
\end{equation*}
with $\mathcal{L}_{M}\colon \widetilde{X}\to X$ an operator representing the numerical reduction technique.
Given $\Phi \in X_M$, $P_M(\Phi)$ is the linear combination of certain basis functions $\phi^{(M)}_0,\dots,\phi^{(M)}_M$ with the coefficients $\Phi_0,\dots,\Phi_M$.
The range of $P_M$ and $\mathcal{L}_{M}$ is thus the space $\Pi_M\coloneqq\spanop\{\phi^{(M)}_0,\dots,\phi^{(M)}_M\}$
, which is a subspace of $X$.
We assume that $\Pi_M\subseteq \Pi_{M+1}$ for all $M\in\mathbb{N}$ (\cref{H-compat-1} in \cref{sec:hypotheses}).
If $h<\tau$, we adopt a piecewise approach for the discretization of $X$.
For simplicity of exposition and to avoid distracting the reader with technical details, from now on we assume $h\geq\tau$, leaving the development of the other case to \cref{sec:h-lt-tau}.%
\footnote{Besides the simplicity of exposition, given that our approach is aimed mainly at studying the stability via the spectrum of $T$, rather than time-integration of the system, considering the case $h\geq\tau$ is in practice not very restrictive, since typically taking a multiple of $h$ gives equivalent information in terms of stability.}

An analogous construction is considered for $X^+$, with spaces $X^+_N$ (identified with $\mathbb{R}^{d(N+1)}$) and $\widetilde{X}^+$, operators $R^+_N$, $P^+_N$, $\mathcal{L}^+_N$ and $\Pi^+_N$ and basis functions $z^{(N)}_0,\dots,z^{(N)}_N$.
In particular,
\begin{equation}\label{PR-RP-N}
R^+_{N} P^+_{N} = I_{X^+_{N}}, \qquad
P^+_{N} R^+_{N} = \mathcal{L}^+_{N},
\end{equation}
The chosen numerical reduction technique may be different from the one used for $X$, thus some compatibility conditions are required (\cref{H-compat-2} in \cref{sec:hypotheses}).
As above, we denote the components of $Z\in X^+_N$ as $(Z_0^T,\dots,Z_N^T)^T$.

The operator $T$ defined in \cref{T} and reformulated as \cref{T-as-V,fixed_point} is reduced to finite dimension as $T_{M, N} \colon X_{M} \to X_{M}$ defined as
\begin{equation}\label{TMN}
T_{M, N} \Phi \coloneqq R_{M} V(P_{M} \Phi, P_{N}^{+} Z^{\ast})_{h},
\end{equation}
where $Z^{\ast} \in X_{N}^{+}$ is a solution of the fixed point equation
\begin{equation}\label{discrete_FP}
Z = R_{N}^{+} \mathcal{F}_{s} V(P_{M} \Phi, P_{N}^{+} Z).
\end{equation}
We study the well-posedness of \cref{discrete_FP} in \cref{sec:fixedpoint}.

Thanks to \cref{decomp_V} and the linearity of the operators involved, $T_{M, N}$ can be rewritten as
\begin{equation*}
T_{M, N} \Phi = T_{M}^{(1)} \Phi + T_{M, N}^{(2)} Z^{\ast},
\end{equation*}
with $T_{M}^{(1)} \colon X_{M} \to X_{M}$ and $T_{M, N}^{(2)} \colon X_{N}^{+} \to X_{M}$ defined as
\begin{equation*}
T_{M}^{(1)} \Phi \coloneqq R_{M} (V^{-} P_{M} \Phi)_{h}, \qquad
T_{M, N}^{(2)} Z \coloneqq R_{M} (V^{+} P_{N}^{+} Z)_{h}.
\end{equation*}
Similarly, the fixed point equation~\cref{discrete_FP} can be rewritten as
\begin{equation*}
(I_{X_{N}^{+}} - U_{N}^{(2)}) Z = U_{M, N}^{(1)} \Phi,
\end{equation*}
with $U_{M, N}^{(1)} \colon X_{M} \to X_{N}^{+}$ and $U_{N}^{(2)} \colon X_{N}^{+} \to X_{N}^{+}$ defined as
\begin{equation*}
U_{M, N}^{(1)} \Phi \coloneqq R_{N}^{+} \mathcal{F}_{s} V^{-} P_{M} \Phi, \qquad
U_{N}^{(2)} Z \coloneqq R_{N}^{+} \mathcal{F}_{s} V^{+} P_{N}^{+} Z.
\end{equation*}
Since the fixed point equation \cref{fixed_point} admits a unique solution in $X^+$ for each $\phi\in X$, the operator $I_{X_{N}^{+}} - U_{N}^{(2)}$ is invertible and the operator $T_{M, N} \colon X_{M} \to X_{M}$ can be reformulated as
\begin{equation*}
T_{M, N} = T_{M}^{(1)} + T_{M, N}^{(2)} (I_{X_{N}^{+}} - U_{N}^{(2)})^{-1} U_{M, N}^{(1)}.
\end{equation*}
This reformulation is convenient to construct the matrix representation of $T_{M, N}$ for the implementation of specific methods (see the references given in \cref{sec:case}).

\section{Convergence analysis}
\label{sec:convergence}

In \cref{sec:hypotheses} we introduce the auxiliary spaces that are needed and we collect the hypotheses that ensure the convergence.
Then in \cref{sec:fixedpoint} we study the fixed-point equation and its discrete counterparts.
Finally, in \cref{sec:convergence-eigenvalues} we present the proof of convergence of the spectra based on the convergence in norm of a suitable auxiliary operator to $T$.

\subsection{Hypotheses and auxiliary spaces}
\label{sec:hypotheses}

The spaces $\widetilde{X}$ and $\widetilde{X}^+$ were considered for the applicability of the chosen reduction technique(s).
However, to ensure that the spectrum of $T_{M,N}$ converges to that of $T$ (in the sense that will be specified in \cref{convergence-theorem}), we need to introduce other assumptions and auxiliary spaces.
The following list collects all these hypotheses.

More specifically, \cref{H-I-FsVp} ensures the well-posedness of the problem; \cref{H-Xhatp} introduces the subspace $\widehat{X}^+$ of $X^+$ on which the functional approximation technique converges in operator norm (see in particular \cref{H-conv}); \cref{H-Xhat} introduces a subspace $\widehat{X}$ of $X$ which is mapped into $\widehat{X}^+$ by the relevant operators; \cref{H-compat-1,H-compat-2} (and also \cref{H-Pi-V-Xhat}) express conditions for the applicability of the discretization techniques and their mutual compatibility.

In particular, \cref{H-FsVp-to-Xhat,H-FsVm-to-Xhat} can be interpreted as the requirement that the delay equation has a regularizing effect on the solutions of the corresponding IVPs; \cref{H-compat-1} ensures that in \cref{TMN,discrete_FP} the operators $R_M$ and $R^+_N$ are applied to functions in $\widetilde{X}$ and $\widetilde{X}^+$, respectively; \cref{H-compat-2} takes care of the translation of the function in \cref{TMN} due to the $h$ subscript.

We recall that we are assuming $h\geq\tau$; alternative version of some hypotheses will be given in \cref{sec:h-lt-tau}.

\begin{H-hyp}
\item\label{H-I-FsVp} The operator $I_{X^{+}} - \mathcal{F}_{s} V^{+} \colon X^+ \to X^+$ is invertible with bounded inverse and~\cref{fixed_point} admits a unique solution in~$X^{+}$.
\item\label{H-Xhatp} There exists a subspace $\widehat{X}^+$ of $\widetilde{X}^+$, equipped with a norm $\norm{\cdot}_{\widehat{X}^{+}}$ that makes it complete, such that:
\begin{H-hyp}
\item\label{H-conv} $\norm{(\mathcal{L}^+_N-I_{X^+})\restriction_{\widehat{X}^+}}_{X^+\leftarrow \widehat{X}^+}\to 0$ as $N\to+\infty$;
\item\label{H-Pi-V-Xhat} $\Pi^+_N\subseteq \widehat{X}^+$ for each $N\in\mathbb{N}$;
\item\label{H-hat-norm} there exists $\hat{c}_1>0$ such that $\norm{\cdot}_{X^{+}} \leq \hat{c}_1 \norm{\cdot}_{\widehat{X}^{+}}$;
\item\label{H-FsVp-to-Xhat} the range of $\mathcal{F}_s V^+\colon X^+\to X^+$ is contained in $\widehat{X}^+$ and $\mathcal{F}_s V^+\colon X^+\to \widehat{X}^+$ is bounded.
\end{H-hyp}
\item\label{H-Xhat} There exists a subspace $\widehat{X}$ of $X$, equipped with a norm $\norm{\cdot}_{\widehat{X}}$ that makes it complete, such that
\begin{H-hyp}
\item\label{H-FsVm-to-Xhat} the range of $\mathcal{F}_s V^-\restriction_{\widehat{X}}\colon \widehat{X}\to X$ is contained in $\widehat{X}^+$ and $\mathcal{F}_s V^-\restriction_{\widehat{X}}\colon \widehat{X}\to \widehat{X}^+$ is bounded;
\item\label{H-Vp-Hhatp-Xhat} $V(\phi,z)_h\in \widehat{X}$ for each $(\phi,z)\in\widehat{X}\times\widehat{X}^+$;
\item\label{H-norm-Vp-z-h} there exists $\hat{c}_2>0$ such that $\norm{(V^+ z)_h}_{\widehat{X}}\leq \hat{c}_2\norm{z}_{X^+}$ for each $z\in\widehat{X}^+$;
\item\label{H-V-Xhat-chain} $V(\phi,z)_h\in \widehat{X}$ for each $(\phi,z)\in X\times\widehat{X}^+$.
\end{H-hyp}
\item\label{H-compat-1} For each $M,N\in\mathbb{N}$, $\Pi_M\subseteq\Pi_{M+1}$, $\Pi^+_N\subseteq\Pi^+_{N+1}$, $\Pi_M\subseteq \widetilde{X}$, $\Pi^+_N\subseteq\widetilde{X}^+$, and, given $(\phi,z) \in \Pi_M\times\Pi^+_N$, $V(\phi,z)_h \in \widetilde{X}$ and $\mathcal{F}_s V(\phi,z) \in \widetilde{X}^+$.
\item\label{H-compat-2} $V(\phi,z)_h\in\Pi_M$ for each $(\phi,z)\in X\times\Pi^+_N$.
\end{H-hyp}

\subsection{The fixed-point equation}
\label{sec:fixedpoint}

We now study the fixed-point equation~\cref{discrete_FP} for a generic function in $X$.
Let $\phi \in X$ and consider the fixed-point equation
\begin{equation}\label{collocation}
Z = R_{N}^{+} \mathcal{F}_{s} V(\phi, P_{N}^{+} Z)
\end{equation}
in $Z \in X_{N}^{+}$.
Our aim is to show that~\cref{collocation} has a unique solution and to study its relation to the unique solution $z^{\ast} \in X^{+}$ of~\cref{fixed_point}.

Observe that $R^+_N$ can be applied in \cref{collocation} only if $\mathcal{F}_{s} V(\phi, P_{N}^{+} Z) \in \widetilde{X}^+$.
In the following we will ensure that this is the case by assuming \cref{H-FsVp-to-Xhat,H-FsVm-to-Xhat} and considering $\phi\in\widehat{X}$.

Using~\cref{decomp_V}, the equations~\cref{fixed_point,collocation} can be rewritten, respectively, as $(I_{X^{+}} - \mathcal{F}_{s} V^{+}) z = \mathcal{F}_{s} V^{-} \phi$ and
\begin{equation}\label{collocation_2}
(I_{X_{N}^{+}} - R_{N}^{+} \mathcal{F}_{s} V^{+} P_{N}^{+}) Z = R_{N}^{+} \mathcal{F}_{s} V^{-} \phi.
\end{equation}
We are thus interested in the invertibility of the operator
\begin{equation}\label{discr_coll_op}
I_{X_{N}^{+}} - R_{N}^{+} \mathcal{F}_{s} V^{+} P_{N}^{+} \colon X_{N}^{+} \to X_{N}^{+}.
\end{equation}
The next result shows its relation to the invertibility of the operator
\begin{equation}\label{cont_coll_op}
I_{X^{+}} - \mathcal{L}_{N}^{+} \mathcal{F}_{s} V^{+} \colon X^{+} \to X^{+}.
\end{equation}

\begin{proposition}\label{discr-cont-coll_eq}
Assume \cref{H-FsVp-to-Xhat}.
If the operator~\cref{cont_coll_op} is invertible, then the operator~\cref{discr_coll_op} is invertible.
Moreover, given $\overline{Z} \in X_{N}^{+}$, the unique solution $\hat{z} \in X^{+}$ of
\begin{equation}\label{cont_coll_eq}
(I_{X^{+}} - \mathcal{L}_{N}^{+} \mathcal{F}_{s} V^{+}) z  = P_{N}^{+} \overline{Z}
\end{equation}
and the unique solution $\widehat{Z} \in X_{N}^{+}$ of
\begin{equation}\label{discr_coll_eq}
(I_{X_{N}^{+}} - R_{N}^{+} \mathcal{F}_{s} V^{+} P_{N}^{+}) Z = \overline{Z}
\end{equation}
are related by $\widehat{Z} = R_{N}^{+} \hat{z}$ and $\hat{z} = P_{N}^{+} \widehat{Z}$.
\end{proposition}
\begin{proof}
Observe that thanks to \cref{H-FsVp-to-Xhat} the range of $\mathcal{F}_{s} V^{+}$ is contained in $\widehat{X}^+\subseteq\widetilde{X}^+$, so the application of $R^+_N$ and $\mathcal{L}^+_{N}$ in \cref{discr_coll_eq,discr_coll_op,cont_coll_eq,cont_coll_op} makes sense.%
\footnote{For \cref{discr_coll_op,discr_coll_eq} only, the same is guaranteed also by \cref{H-compat-1}, which indeed implies that $\mathcal{F}_{s} V^{+} P_{N}^{+} Z \in \widetilde{X}^+$.}
If~\cref{cont_coll_op} is invertible, then, given $\overline{Z} \in X_{N}^{+}$, \cref{cont_coll_eq} has a unique solution, say $\hat{z} \in X^{+}$.
Then, by~\cref{PR-RP-N},
\begin{equation}\label{hat_eq_1}
\hat{z} = P_{N}^{+}(R_{N}^{+} \mathcal{F}_{s} V^{+} \hat{z} + \overline{Z})
\end{equation}
and
\begin{equation}\label{hat_eq_2}
R_{N}^{+} \hat{z} = R_{N}^{+} \mathcal{F}_{s} V^{+} \hat{z} + \overline{Z}
\end{equation}
hold.
Hence, by substituting~\cref{hat_eq_2} in~\cref{hat_eq_1},
\begin{equation}\label{hat_eq_3}
\hat{z} = P_{N}^{+} R_{N}^{+} \hat{z}
\end{equation}
and, by substituting~\cref{hat_eq_3} in~\cref{hat_eq_2}, $R_{N}^{+} \hat{z} = R_{N}^{+} \mathcal{F}_{s} V^{+} P_{N}^{+} R_{N}^{+} \hat{z} + \overline{Z}$, i.e., $R_{N}^{+} \hat{z}$ is a solution of~\cref{discr_coll_eq}.

Vice versa, if $\widehat{Z} \in X_{N}^{+}$ is a solution of~\cref{discr_coll_eq}, then $P_{N}^{+} \widehat{Z} = \mathcal{L}_{N}^{+} \mathcal{F}_{s} V^{+} P_{N}^{+} \widehat{Z} + P_{N}^{+} \overline{Z}$ holds again by~\cref{PR-RP-N}, i.e., $P_{N}^{+} \widehat{Z}$ is a solution of~\cref{cont_coll_eq}.
Hence, by uniqueness, $\hat{z} = P_{N}^{+} \widehat{Z}$ holds.

Finally, if $\widehat{Z}_{1}, \widehat{Z}_{2} \in X_{N}^{+}$ are solutions of~\cref{discr_coll_eq}, then $P_{N}^{+} \widehat{Z}_{1} = \hat{z} = P_{N}^{+} \widehat{Z}_{2}$ and, once again by~\cref{PR-RP-N}, $\widehat{Z}_{1} = R_{N}^{+} P_{N}^{+} \widehat{Z}_{1} = R_{N}^{+} P_{N}^{+} \widehat{Z}_{2} = \widehat{Z}_{2}$.
Therefore $\widehat{Z} \coloneqq R_{N}^{+} \hat{z}$ is the unique solution of~\cref{discr_coll_eq} and the operator~\cref{discr_coll_op} is invertible.
\end{proof}

\begin{corollary}\label{coll_hat_corollary}
Assume \cref{H-FsVp-to-Xhat,H-FsVm-to-Xhat}.
If the operator \cref{cont_coll_op} is invertible, then for each $\phi\in\widehat{X}$
\begin{equation}\label{collocation_hat}
z = \mathcal{L}_{N}^{+} \mathcal{F}_{s} V(\phi, z)
\end{equation}
has a unique solution $w^* \in X^+$, \cref{collocation} has a unique solution $Z^*\in X^+_N$, and $Z^*=R^+_N w^*$ and $w^*=P^+_N Z^*$ hold.
\end{corollary}
\begin{proof}
Observe that thanks to \cref{H-FsVm-to-Xhat} $\mathcal{F}_{s} V^{-} \phi \in \widehat{X}^+\subseteq\widetilde{X}^+$; together with the consequences of \cref{H-FsVp-to-Xhat} mentioned in the previous proof, this implies that the application of $R^+_N$ and $\mathcal{L}^+_{N}$ in \cref{collocation_hat,collocation,collocation_2} makes sense.

As observed above, \cref{collocation} is equivalent to~\cref{collocation_2}, hence, by choosing
\begin{equation}\label{Wbar}
\overline{Z} = R_{N}^{+} \mathcal{F}_{s} V^{-} \phi,
\end{equation}
it is equivalent to~\cref{discr_coll_eq}.
Observe also that, thanks to~\cref{PR-RP-N}, \cref{collocation_hat} can be rewritten as $(I_{X^{+}} - \mathcal{L}_{N}^{+} \mathcal{F}_{s} V^{+}) z
= \mathcal{L}_{N}^{+} \mathcal{F}_{s} V^{-} \phi
= P_{N}^{+} R_{N}^{+} \mathcal{F}_{s} V^{-} \phi$,
which is equivalent to~\cref{cont_coll_eq} with the choice~\cref{Wbar}.
Thus, by \cref{discr-cont-coll_eq}, if the operator~\cref{cont_coll_op} is invertible, then the equation~\cref{collocation} has a unique solution $Z^{\ast} \in X_{N}^{+}$ such that
\begin{equation}\label{stars-vs-hats}
Z^{\ast} = R_{N}^{+} w^{\ast}, \qquad
w^{\ast} = P_{N}^{+} Z^{\ast},
\end{equation}
where $w^{\ast} \in X^{+}$ is the unique solution of~\cref{collocation_hat}.
Note for clarity that~\cref{Wbar} implies $w^{\ast} = \hat{z}$ for~$\hat{z}$ in \cref{discr-cont-coll_eq}.
\end{proof}

The next result shows that~\cref{cont_coll_op} is invertible under suitable assumptions.

\begin{proposition}\label{cont_coll_op-invertible}
If \cref{H-I-FsVp,H-FsVp-to-Xhat,H-conv,H-FsVm-to-Xhat} hold, then there exists a positive integer $N_{0}$ such that, for any $N \geq N_{0}$, the operator~\cref{cont_coll_op} is invertible and
\begin{equation*}
\norm{(I_{X^{+}} - \mathcal{L}_{N}^{+} \mathcal{F}_{s} V^{+})^{- 1}}_{X^{+} \leftarrow X^{+}} \leq 2 \norm{(I_{X^{+}} - \mathcal{F}_{s} V^{+})^{- 1}}_{X^{+} \leftarrow X^{+}}.
\end{equation*}
Moreover, for each $\phi \in \widehat{X}$, \cref{collocation_hat} has a unique solution $w^{\ast} \in X^{+}$ and
\begin{equation*}
\norm{w^{\ast} - z^{\ast}}_{X^{+}} \leq 2 \norm{(I_{X^{+}} - \mathcal{F}_{s} V^{+})^{- 1}}_{X^{+} \leftarrow X^{+}} \norm{\mathcal{L}_{N}^{+} z^{\ast} - z^{\ast}}_{X^{+}},
\end{equation*}
where $z^{\ast} \in X^{+}$ is the unique solution of~\cref{fixed_point}.
\end{proposition}
\begin{proof}
In this proof, let $I \coloneqq I_{X^{+}}$.
By \cref{H-FsVp-to-Xhat,H-conv},
\begin{equation*}
\norm{(\mathcal{L}_{N}^{+} - I) \mathcal{F}_{s} V^{+}}_{X^{+} \leftarrow X^{+}} \leq \norm{(\mathcal{L}^+_N-I)\restriction_{\widehat{X}^+}}_{X^+\leftarrow \widehat{X}^+} \norm{\mathcal{F}_{s} V^{+}}_{\widehat{X}^{+}\leftarrow X^{+}} \xrightarrow[N \to \infty]{} 0.
\end{equation*}
In particular, there exists a positive integer $N_{0}$ such that, for each integer $N \geq N_{0}$,
\begin{equation*}
\norm{(\mathcal{L}_{N}^{+} - I) \mathcal{F}_{s} V^{+}}_{X^{+} \leftarrow X^{+}} \leq \frac{1}{2 \norm{(I - \mathcal{F}_{s} V^{+})^{- 1}}_{X^{+} \leftarrow X^{+}}},
\end{equation*}
i.e., $\norm{(\mathcal{L}_{N}^{+} - I) \mathcal{F}_{s} V^{+}}_{X^{+} \leftarrow X^{+}} \norm{(I - \mathcal{F}_{s} V^{+})^{- 1}}_{X^{+} \leftarrow X^{+}} \leq \frac{1}{2}$,
which holds since $I - \mathcal{F}_{s} V^{+}$ is invertible with bounded inverse by virtue of \cref{H-I-FsVp}.

Considering the operator $I - \mathcal{L}_{N}^{+} \mathcal{F}_{s} V^{+}$ as a perturbed version of $I - \mathcal{F}_{s} V^{+}$ and writing $I - \mathcal{L}_{N}^{+} \mathcal{F}_{s} V^{+} = I - \mathcal{F}_{s} V^{+} - (\mathcal{L}_{N}^{+} - I) \mathcal{F}_{s} V^{+}$, by the Banach perturbation lemma \cite[Theorem 10.1]{Kress2014}, for each integer $N \geq N_{0}$ the operator $I - \mathcal{L}_{N}^{+} \mathcal{F}_{s} V^{+}$ is invertible and
\begin{equation*}
\begin{split}
\norm{(I - \mathcal{L}_{N}^{+} \mathcal{F}_{s} V^{+})^{- 1}}_{X^{+} \leftarrow X^{+}}
&\leq \frac{\norm{(I - \mathcal{F}_{s} V^{+})^{- 1}}_{X^{+} \leftarrow X^{+}}}{1 - \norm{(I - \mathcal{F}_{s} V^{+})^{- 1} ((\mathcal{L}_{N}^{+} - I) \mathcal{F}_{s} V^{+})}_{X^{+} \leftarrow X^{+}}} \\
&\leq 2 \norm{(I - \mathcal{F}_{s} V^{+})^{- 1}}_{X^{+} \leftarrow X^{+}}.
\end{split}
\end{equation*}
Hence, fixed $\phi \in \widehat{X}$, \cref{collocation_hat} has a unique solution $w^{\ast} \in X^{+}$.
For the same $\phi$, let $e^{\ast} \in X^{+}$ such that $w^{\ast} = z^{\ast} + e^{\ast}$, where $z^{\ast} \in X^{+}$ is the unique solution of~\cref{fixed_point}.
Then $z^{\ast} + e^{\ast}
= \mathcal{L}_{N}^{+} \mathcal{F}_{s} V(\phi, z^{\ast} + e^{\ast})
= \mathcal{L}_{N}^{+} \mathcal{F}_{s} V(\phi, z^{\ast}) + \mathcal{L}_{N}^{+} \mathcal{F}_{s} V^{+} e^{\ast} 
= \mathcal{L}_{N}^{+} z^{\ast}  + \mathcal{L}_{N}^{+} \mathcal{F}_{s} V^{+} e^{\ast}$ 
and $(I - \mathcal{L}_{N}^{+} \mathcal{F}_{s} V^{+}) e^{\ast} = (\mathcal{L}_{N}^{+} - I) z^{\ast}$, completing the proof.
\end{proof}

\subsection{Convergence of the eigenvalues}
\label{sec:convergence-eigenvalues}

We introduce a finite-rank infinite-dimensional operator $\widehat{T}_{M, N}$ associated to $T_{M, N}$ and show the relation between their spectra.

\begin{proposition}\label{TMN-hTMN-same}
Assume \cref{H-compat-1}.
The operator $T_{M, N}$ has the same nonzero eigenvalues, with the same geometric and partial multiplicities, of the operator
\begin{equation*}
\widehat{T}_{M, N} \coloneqq P_{M} T_{M, N} R_{M} \colon \widetilde{X} \to \widetilde{X}.
\end{equation*}
Moreover, if $\Phi \in X_{M}$ is an eigenvector of $T_{M, N}$ associated to a nonzero eigenvalue~$\mu$, then $P_{M} \Phi \in \widetilde{X}$ is an eigenvector of $\widehat{T}_{M, N}$ associated to the same eigenvalue $\mu$.
\end{proposition}
\begin{proof}
Apply \cite[Proposition 4.1]{BredaMasetVermiglio2012}.
\end{proof}

We define the operator $\widehat{T}_{N} \colon \widehat{X} \to \widehat{X}$ as
\begin{equation*}
\widehat{T}_{N} \phi \coloneqq V(\phi, w^{\ast})_{h},
\end{equation*}
where $w^{\ast} \in X^{+}$ is the solution of the fixed point equation~\cref{collocation_hat}, which, under \cref{H-I-FsVp,H-FsVp-to-Xhat,H-conv,H-FsVm-to-Xhat}, is unique thanks to \cref{coll_hat_corollary,cont_coll_op-invertible}, for $N\geq N_0$ with $N_0$ given by \cref{cont_coll_op-invertible}.
Observe that $w^{\ast}\in \Pi^+_N\subseteq\widetilde{X}^+$, thanks to \cref{H-compat-1}.
For $\phi \in \widetilde{X}$, by~\cref{stars-vs-hats},
\begin{equation*}
\begin{split}
\widehat{T}_{M, N} \phi
&= P_{M} T_{M, N} R_{M} \phi \\
&= P_{M} R_{M} V(P_{M} R_{M} \phi, P_{N}^{+} Z^{\ast})_{h} \\
&= \mathcal{L}_{M} V(\mathcal{L}_{M} \phi, w^{\ast})_{h} \\
&= \mathcal{L}_{M} \widehat{T}_{N} \mathcal{L}_{M} \phi,
\end{split}
\end{equation*}
where $Z^{\ast} \in X_{N}^{+}$ and $w^{\ast} \in \widetilde{X}^{+}$ are the solutions, respectively, of~\cref{discrete_FP} applied to $\Phi = R_{M} \phi$ and of~\cref{collocation_hat} with $\mathcal{L}_{M} \phi$ replacing $\phi$.
These solutions are unique under \cref{H-I-FsVp,H-FsVp-to-Xhat,H-conv,H-FsVm-to-Xhat} for $N\geq N_0$, thanks to \cref{discr-cont-coll_eq,coll_hat_corollary,cont_coll_op-invertible}.

Now we show the relation between the spectra of $\widehat{T}_{M, N}$ and $\widehat{T}_{N}$.

\begin{proposition}\label{hTMN-hTN-same}
Assume \cref{H-I-FsVp,H-FsVp-to-Xhat,H-conv,H-FsVm-to-Xhat,H-compat-1}, let $N_{0}$ be given by \cref{cont_coll_op-invertible}, let $N\geq N_0$ and assume that \cref{H-compat-2} holds for $M$ and $N$.
The operator $\widehat{T}_{M, N}$ has the same nonzero eigenvalues, with the same geometric and partial multiplicities and associated eigenvectors, of the operator $\widehat{T}_{N}$.
\end{proposition}
\begin{proof}
Let $\phi\in\widehat{X}$ and let $w^{\ast} \in X^{+}$ be the unique solution of~\cref{collocation_hat}.
Observe that $w^*\in\Pi^+_N$ and that $\widehat{T}_{N} \phi\in\Pi_M$ from \cref{H-compat-2}.
Thus, $\widehat{T}_{N}$ and $\widehat{T}_{M, N} = \mathcal{L}_{M} \widehat{T}_{N} \mathcal{L}_{M}$ have range contained in $\Pi_{M}$.
By~\cite[Proposition 4.3 and Remark 4.4]{BredaMasetVermiglio2012}, $\widehat{T}_{N}$ and $\widehat{T}_{M, N}$ have the same nonzero eigenvalues, with the same geometric and partial multiplicities and associated eigenvectors, as their restrictions to~$\Pi_{M}$.
Observing that $\widehat{T}_{M, N} \restriction_{\Pi_{M}} = \mathcal{L}_{M} \widehat{T}_{N} \mathcal{L}_{M} \restriction_{\Pi_{M}} = \widehat{T}_{N} \restriction_{\Pi_{M}}$, the thesis follows.
\end{proof}

\begin{lemma}\label{LNpFXsVXp-invertible}
If \cref{H-I-FsVp,H-FsVp-to-Xhat,H-hat-norm} hold, then $(I_{X^{+}} - \mathcal{F}_{s} V^{+}) \restriction_{\widehat{X}^{+}}$ is invertible with bounded inverse.
\end{lemma}
\begin{proof}
Since $I_{X^{+}} - \mathcal{F}_{s} V^{+}$ is invertible with bounded inverse by virtue of \cref{H-I-FsVp}, given $f \in \widehat{X}^{+}$ the equation $(I_{X^{+}} - \mathcal{F}_{s} V^{+}) z = f$ has a unique solution $z \in X^{+}$, which by \cref{H-FsVp-to-Xhat} is in $\widehat{X}^{+}$.
Hence, the operator $(I_{X^{+}} - \mathcal{F}_{s} V^{+}) \restriction_{\widehat{X}^{+}}$ is invertible.
It is also bounded thanks to \cref{H-hat-norm}:
\begin{equation*}
\norm{\mathcal{F}_{s} V^{+} \restriction_{\widehat{X}^{+}}}_{\widehat{X}^{+} \leftarrow \widehat{X}^{+}} \leq \hat{c}_1 \norm{\mathcal{F}_{s} V^{+}}_{\widehat{X}^{+} \leftarrow X^{+}}.
\end{equation*}
The bounded inverse theorem completes the proof.
\end{proof}

We now prove the norm convergence of $\widehat{T}_{N}$ to $T$, which is the key step to obtain the final convergence theorem.

\begin{proposition}\label{norm-convergence}
If \cref{H-I-FsVp,H-Xhatp,H-FsVm-to-Xhat,H-Vp-Hhatp-Xhat,H-norm-Vp-z-h} hold, then $\norm{\widehat{T}_{N} - T\restriction_{\widehat{X}}}_{\widehat{X} \leftarrow \widehat{X}} \to 0$ for $N \to \infty$.
\end{proposition}
\begin{proof}
Let $\phi \in \widehat{X}$ and let $z^{\ast}$ and $w^{\ast}$ be the solutions of the fixed point equations~\cref{fixed_point,collocation_hat}, respectively.
Recall that $w^{\ast}\in \Pi^+_N \subseteq \widehat{X}^+$ thanks to \cref{H-Pi-V-Xhat}.
Assuming \cref{H-FsVp-to-Xhat,H-FsVm-to-Xhat} and recalling that $z^{\ast} = \mathcal{F}_{s} V^{+} z^{\ast} + \mathcal{F}_{s} V^{-} \phi$, it is clear that $z^{\ast} \in \widehat{X}^{+}$.
Then $(\widehat{T}_{N} - T) \phi 
= V(\phi, w^{\ast})_{h} - V(\phi, z^{\ast})_{h}
= V^{+}(w^{\ast} - z^{\ast})_{h} \in \widehat{X}$, thanks to \cref{H-Vp-Hhatp-Xhat} (applied with $\phi=0_X$).
Moreover, by \cref{H-norm-Vp-z-h} there exists $\hat{c}_2>0$ such that
$\norm{(\widehat{T}_{N} - T) \phi}_{\widehat{X}}
= \norm{V^{+}(w^{\ast} - z^{\ast})_{h}}_{\widehat{X}} \leq \hat{c}_2 \norm{w^{\ast} - z^{\ast}}_{X^+}$.
Assuming also \cref{H-I-FsVp,H-conv}, by \cref{cont_coll_op-invertible}, there exists a positive integer $N_{0}$ such that, for any $N \geq N_{0}$,
\begin{equation*}
\begin{aligned}
\norm{w^{\ast} - z^{\ast}}_{X^{+}}
&\leq 2 \norm{(I_{X^{+}} - \mathcal{F}_{s} V^{+})^{- 1}}_{X^{+} \leftarrow X^{+}} \norm{\mathcal{L}_{N}^{+} z^{\ast} - z^{\ast}}_{X^{+}} \\
&\leq 2 \norm{(I_{X^{+}} - \mathcal{F}_{s} V^{+})^{- 1}}_{X^{+} \leftarrow X^{+}} \norm{(\mathcal{L}_{N}^{+} - I_{X^{+}}) \restriction_{\widehat{X}^{+}}}_{X^{+} \leftarrow \widehat{X}^{+}} \norm{z^{\ast}}_{\widehat{X}^{+}}.
\end{aligned}
\end{equation*}
Finally,
\begin{equation*}
\norm{z^{\ast}}_{\widehat{X}^{+}} \leq \norm{((I_{X^{+}} - \mathcal{F}_{s} V^{+}) \restriction_{\widehat{X}^{+}})^{- 1}}_{\widehat{X}^{+} \leftarrow \widehat{X}^{+}} \norm{\mathcal{F}_{s} V^{-}}_{\widehat{X}^{+} \leftarrow \widehat{X}} \norm{\phi}_{\widehat{X}}
\end{equation*}
thanks to \cref{H-FsVm-to-Xhat} and \cref{LNpFXsVXp-invertible} (requiring \cref{H-hat-norm}).
Putting everything together, we obtain
\begin{equation*}
\norm{(\widehat{T}_{N} - T) \phi}_{\widehat{X}}
\leq \text{(positive constant)}\cdot \norm{(\mathcal{L}_{N}^{+} - I_{X^{+}}) \restriction_{\widehat{X}^{+}}}_{X^{+} \leftarrow \widehat{X}^{+}} \norm{\phi}_{\widehat{X}}
\end{equation*}
and \cref{H-conv} completes the proof.
\end{proof}

The following lemma summarizes some tools from \cite{Chatelin2011} which are the basis of the final convergence results.
Recall that the \emph{ascent} of an eigenvalue is the maximum length of an associated Jordan chain.

\begin{lemma}\label{Chatelin-norm-convergence}
Let $U$ be a Banach space, $A$ a linear and bounded operator on $U$ and $\{A_{N}\}_{N \in \mathbb{N}}$ a sequence of linear and bounded operators on $U$ such that $\norm{A_{N} - A}_{U \leftarrow U} \to 0$ for $N \to \infty$.
If $\mu \in \mathbb{C}$ is an eigenvalue of $A$ with finite algebraic multiplicity $\nu$ and ascent $l$, and $\Delta$ is a neighborhood of $\mu$ such that $\mu$ is the only eigenvalue of $A$ in $\Delta$, then there exists a positive integer $\overline{N}$ such that, for any $N \geq \overline{N}$, $A_{N}$ has in $\Delta$ exactly $\nu$ eigenvalues $\mu_{N, j}$, $j \in \{1, \dots, \nu\}$, counting their multiplicities.
Moreover, by setting $\epsilon_{N} \coloneqq \norm{(A_{N} - A) \restriction_{\mathcal{E}_{\mu}}}_{U \leftarrow \mathcal{E}_{\mu}}$, where $\mathcal{E}_{\mu}$ is the generalized eigenspace of $\mu$ equipped with the norm $\norm{\cdot}_{U}$ restricted to $\mathcal{E}_{\mu}$, the following holds:
\begin{equation*}
\max_{j \in \{1, \dots, \nu\}} \abs{\mu_{N, j} - \mu} = O(\epsilon_{N}^{1 / l}).
\end{equation*}
\end{lemma}
\begin{proof}
By \cite[Example 3.8 and Theorem 5.22]{Chatelin2011}, the norm convergence of $A_{N}$ to $A$ implies the strongly stable convergence $A_{N} - \mu I_{U} \xrightarrow{ss} A - \mu I_{U}$ for all $\mu$ in the resolvent set of $A$ and all isolated eigenvalues $\mu$ of finite multiplicity of $A$.
The thesis follows then by \cite[Proposition 5.6 and Theorem 6.7]{Chatelin2011}.
\end{proof}

Since \cref{norm-convergence} proves the norm convergence of $\widehat{T}_N$ to the restriction of $T$ to $\widehat{X}$, we also need the following result.
Recall that the length of any Jordan chain is a \emph{partial multiplicity} of the corresponding eigenvalue.

\begin{proposition}\label{T-hTN-restriction-to-Xhat}
Assume \cref{H-I-FsVp,H-FsVp-to-Xhat,H-hat-norm,H-FsVm-to-Xhat,H-Vp-Hhatp-Xhat,H-V-Xhat-chain}.
The operator $T$ has the same nonzero eigenvalues, with the same geometric and partial multiplicities and associated eigenvectors, as its restriction to~$\widehat{X}$.
\end{proposition}
\begin{proof}
Let $\phi\in X$ and let $z^*\in X^+$ be the unique solution of \cref{fixed_point}, thanks to \cref{H-I-FsVp}.

If $\phi\in\widehat{X}$, assuming also \cref{H-I-FsVp,H-FsVp-to-Xhat,H-hat-norm}, by \cref{H-FsVm-to-Xhat,LNpFXsVXp-invertible}, $z^*=(I_{X^+}-\mathcal{F}_s V^+)^{-1}\mathcal{F}_s V^- \phi \in \widehat{X}^+$.
Hence, from \cref{H-Vp-Hhatp-Xhat}, $T\phi = V(\phi,z^*)_h \in \widehat{X}$.

Let now $\psi\in \widehat{X}$ and let $\lambda\in\mathbb{C}\setminus\{0\}$ be an eigenvalue of $T$.
Assuming $(\lambda I_X - T)\phi=\psi$, we want to prove that $\phi\in\widehat{X}$.
Observe that $(\lambda I_X - T)\phi=\lambda\phi-T\phi=\lambda\phi-V(\phi,z^*)_h$, so
\begin{equation}\label{eq:phi-psi}
\phi=\lambda^{-1}(\psi+V(\phi,z^*)_h).
\end{equation}
\Cref{H-V-Xhat-chain} gives $V(\phi,z^*)_h\in\widehat{X}$ and $\phi\in\widehat{X}$ follows from \cref{eq:phi-psi}.

The thesis follows from \cite[Proposition 4.3]{BredaMasetVermiglio2012}).
%
\end{proof}

\begin{proposition}\label{eig-convergence}
Assume that \cref{H-I-FsVp,H-Xhatp,H-Xhat} hold and let $N_{0}$ be given by \cref{cont_coll_op-invertible}.
If $\mu \in \mathbb{C} \setminus \{0\}$ is an eigenvalue of $T$ with finite algebraic multiplicity $\nu$ and ascent $l$, and $\Delta$ is a neighborhood of $\mu$ such that $\mu$ is the only eigenvalue of $T$ in $\Delta$, then there exists a positive integer $N_{1} \geq N_{0}$ such that, for any $N \geq N_{1}$, $\widehat{T}_{N}$ has in $\Delta$ exactly $\nu$ eigenvalues $\mu_{N, j}$, $j \in \{1, \dots, \nu\}$, counting their multiplicities.
Moreover, by setting $\epsilon_{N} \coloneqq \norm{(\widehat{T}_{N} - T) \restriction_{\mathcal{E}_{\mu}}}_{\widehat{X} \leftarrow \mathcal{E}_{\mu}}$, where $\mathcal{E}_{\mu}$ is the generalized eigenspace of $\mu$ equipped with the norm $\norm{\cdot}_{\widehat{X}}$ restricted to $\mathcal{E}_{\mu}$, the following holds:
\begin{equation*}
\max_{j \in \{1, \dots, \nu\}} \abs{\mu_{N, j} - \mu} = O(\epsilon_{N}^{1 / l}).
\end{equation*}
Finally, $\epsilon_N = O(\norm{(\mathcal{L}_{N}^{+} - I_{X^{+}}) \restriction_{\widehat{X}^{+}}}_{X^{+} \leftarrow \widehat{X}^{+}})$.
\end{proposition}
\begin{proof}
By \cref{norm-convergence}, $\norm{\widehat{T}_{N}\restriction_{\widehat{X}} - T\restriction_{\widehat{X}}}_{\widehat{X} \leftarrow \widehat{X}} \to 0$ for $N \to \infty$.
The first part of the thesis is obtained by applying \cref{Chatelin-norm-convergence,T-hTN-restriction-to-Xhat}.
Observe that $\mathcal{E}_{\mu}\subseteq\widehat{X}$ from \cref{T-hTN-restriction-to-Xhat}.

Let $\{\phi_{1}, \dots, \phi_{\nu}\}$ be a basis of $\mathcal{E}_{\mu}$.
An element $\phi\in\mathcal{E}_{\mu}$ can be written as $\phi = \sum_{j = 1}^{\nu} \alpha_{j}(\phi) \phi_{j}$, with $\alpha_{j}(\phi) \in \mathbb{C}$ for $j \in \{1, \dots, \nu\}$, hence
\begin{equation*}
\norm{(\widehat{T}_{N} - T) \phi}_{\widehat{X}} \leq \max_{j \in \{1, \dots, \nu\}} \abs{\alpha_{j}(\phi)} \cdot \sum_{j = 1}^{\nu} \norm{(\widehat{T}_{N} - T) \phi_{j}}_{X}.
\end{equation*}
The function $\phi \mapsto \max_{j \in \{1, \dots, \nu\}} \abs{\alpha_{j}(\phi)}$ is a norm on $\mathcal{E}_{\mu}$, so it is equivalent to the norm of $\widehat{X}$ restricted to $\mathcal{E}_{\mu}$.
Thus, in particular, there exists a positive constant $c$ independent of $\phi$ such that
$\max_{j \in \{1, \dots, \nu\}} \abs{\alpha_{j}(\phi)} \leq c \norm{\phi}_{X}$
and
\begin{equation*}
\epsilon_{N} = \norm{(\widehat{T}_{N} - T) \restriction_{\mathcal{E}_{\mu}}}_{\widehat{X} \leftarrow \mathcal{E}_{\mu}} \leq c \sum_{j = 1}^{\nu} \norm{(\widehat{T}_{N} - T) \phi_{j}}_{\widehat{X}}.
\end{equation*}
Moreover, for $j \in \{1, \dots, \nu\}$, as seen in the proof of~\cref{norm-convergence},
\begin{equation}\label{eq:estimate}
\begin{aligned}
\norm{(\widehat{T}_{N} - T) \phi_{j}}_{\widehat{X}}
&\leq 2 \hat{c}_2 \norm{(I_{X^{+}} - \mathcal{F}_{s} V^{+})^{- 1}}_{X^{+} \leftarrow X^{+}} \norm{(\mathcal{L}_{N}^{+} - I_{X^{+}}) z^{\ast}_j}_{X^{+}}, \\
&\leq 2 \hat{c}_2 \norm{(I_{X^{+}} - \mathcal{F}_{s} V^{+})^{- 1}}_{X^{+} \leftarrow X^{+}} \norm{(\mathcal{L}_{N}^{+} - I_{X^{+}}) \restriction_{\widehat{X}^{+}}}_{X^{+} \leftarrow \widehat{X}^{+}} \norm{z^{\ast}_j}_{\widehat{X}^{+}},
\end{aligned}
\end{equation}
where $z^{\ast}_{j}\in\widehat{X}^+$ is the solution of~\cref{fixed_point} associated to $\phi_{j}$.
\end{proof}

\begin{theorem}
\label{convergence-theorem}
Assume that \cref{H-I-FsVp,H-Xhatp,H-Xhat,H-compat-1} hold, let $N_{0}$ be given by \cref{cont_coll_op-invertible} and assume that there exists $N_1\geq N_0$ such that \cref{H-compat-2} holds for all $N\geq N_1$ and for a chosen $M$.
If $\mu \in \mathbb{C} \setminus \{0\}$ is an eigenvalue of $T$ with finite algebraic multiplicity $\nu$ and ascent $l$, and $\Delta$ is a neighborhood of $\mu$ such that $\mu$ is the only eigenvalue of $T$ in $\Delta$, then there exists a positive integer $N_{2} \geq N_{1}$ such that, for any $N \geq N_{2}$, $T_{M,N}$ has in $\Delta$ exactly $\nu$ eigenvalues $\mu_{M,N,j}$, $j \in \{1, \dots, \nu\}$, counting their multiplicities.
Moreover,
\begin{equation}\label{eq:final-error}
\max_{j \in \{1, \dots, \nu\}} \abs{\mu_{M,N, j} - \mu} = O(\epsilon_{N}^{1 / l})
\end{equation}
with $\epsilon_N = O(\norm{(\mathcal{L}_{N}^{+} - I_{X^{+}}) \restriction_{\widehat{X}^{+}}}_{X^{+} \leftarrow \widehat{X}^{+}})$.
\end{theorem}
\begin{proof}
If $N\geq N_0$ and \cref{H-compat-2} holds for $M$ and $N$, by \cref{TMN-hTMN-same,hTMN-hTN-same} the operators $T_{M, N}$, $\widehat{T}_{M, N}$ and $\widehat{T}_{N}$ have the same nonzero eigenvalues, with the same geometric and partial multiplicities and associated eigenvectors.
The thesis follows by \cref{eig-convergence}.
\end{proof}

\section{Case studies}
\label{sec:case}

In \cref{sec:pscoll-rfde} and \cref{sec:pscoll-re} we revisit \cite{BredaMasetVermiglio2012,BredaMasetVermiglio2015} and \cite{BredaLiessi2018}, respectively, in light of this work, describing the operators $V$ and $\mathcal{F}_s$, the auxiliary spaces and the pseudospectral collocation methods, and verifying the hypotheses of \cref{sec:hypotheses}; we also summarize the results on the order of convergence based on the regularity of the relevant eigenfunctions.
In \cref{sec:pscoll-pw} we provide the convergence proof for the piecewise approach of \cite{BredaLiessiVermiglio2022} for the first time by applying the proposed framework.
Finally in \cref{sec:weightedresiduals} we provide the convergence proof also for a method of weighted residuals as an example of a method from the literature lacking a formal proof.

\subsection{Pseudospectral collocation of RFDEs}
\label{sec:pscoll-rfde}


For $s \in \mathbb{R}$, we consider the IVP defined by \cref{RFDE} for $t\geq s$ and $x_s = \phi \in X$.
It is a classical result (see, e.g., \cite[Theorems 2.2.1, 2.2.2 and 2.2.3]{HaleVerduynLunel1993} and \cite[Theorem 3.7 and Remark 3.8]{Smith2011}) that local Lipschitz continuity of the right-hand side of the RFDE implies the existence and uniqueness of solutions of the IVP and their continuous dependence on the initial data; moreover, the solutions are globally defined in the future if the Lipschitz continuity is global.
In particular, if $t \mapsto \norm{L(t)}_{\mathbb{R}^{d} \leftarrow X}$ is bounded (as is the case, e.g., if $L(t)$ is periodic), for each $s \in \mathbb{R}$ and $\phi \in X$ the IVP admits a unique solution on $[s-\tau, +\infty)$.

For $t\geq s$, the evolution operator $U(t,s)\colon X\to X$ is defined as $U(t, s) \phi = x_{t}(\cdot; s, \phi)$, where $x(t; s, \phi)$ is the solution of the IVP.

As the spaces $X^+$ and $X^{\pm}$ we consider $C([0, h], \mathbb{R}^{d})$ and $C([-\tau, h], \mathbb{R}^{d})$, respectively, with the norms $\norm{\cdot}_{X^{+}}$ and $\norm{\cdot}_{X^{\pm}}$ being the corresponding uniform norms.

We define the operator $V \colon X \times X^{+} \to X^{\pm}$ as
\begin{equation*}
V(\phi, z)(t) \coloneqq
\begin{cases}
\phi(0) + \displaystyle\int_{0}^{t} z(\sigma) \dd \sigma, & \quad t \in (0, h], \\
\phi(t), & \quad t \in [-\tau, 0],
\end{cases}
\end{equation*}
and the operator $\mathcal{F}_{s} \colon X^{\pm} \to X^{+}$ as
\begin{equation}\label{Fs_RFDE}
\mathcal{F}_{s} u(t) \coloneqq L(s+t) u_t, \qquad t \in [0, h].
\end{equation}

Assuming that $t \mapsto \norm{L(t)}_{\mathbb{R}^{d} \leftarrow X}$ is bounded, it is easy to prove that \cref{H-I-FsVp} is satisfied (see \cite[end of section 2]{BredaMasetVermiglio2012}).

We choose $\widetilde{X}=X$ and $\widetilde{X}^+=X^+$, since point-wise evaluation does not require more than continuity.

If $h \geq \tau$, we choose interpolation nodes $\theta_{M, 0}, \dots, \theta_{M, M}$ in $[-\tau, 0]$ with%
\footnote{In principle, it is not necessary to require that the endpoints of the interval are among the nodes.
However, including the endpoints of is convenient since very often the right-hand side uses the values of the unknown function at the current time and at the maximum delay: with such nodes, interpolation can be avoided in those cases.}
\begin{equation*}
- \tau = \theta_{M, M} < \dots < \theta_{M, 0} = 0,
\end{equation*}
$X_{M} \coloneqq \mathbb{R}^{d (M + 1)}$ and define $R_M$ as, for $\phi\in X$,
\begin{equation*}
R_{M} \phi \coloneqq (\phi(\theta_{M, 0})^T, \dots, \phi(\theta_{M, M})^T)^T.
\end{equation*}
We define $P_{M}$ as the discrete Lagrange interpolation operator
\begin{equation*}
(P_{M} \Phi)(\theta) \coloneqq \sum_{m = 0}^{M} \ell_{M, m}(\theta) \Phi_{m}, \qquad \theta \in [-\tau, 0],
\end{equation*}
where $\{\ell_{M, 0}, \dots, \ell_{M, M}\}$ is the Lagrange basis for the chosen nodes and $\Phi_m\in\mathbb{R}^{d}$ is the $m$-th block of $d$ components of $\Phi$.
The space $\Pi_M$ is the space of polynomials of degree at most $M$ on $[-\tau,0]$.
For the case $h<\tau$, see \cite{BredaMasetVermiglio2012,BredaMasetVermiglio2015}.

We also choose interpolation nodes $t_{N, 1}, \dots, t_{N, N}$ in $[0,t]$ with
\begin{equation*}
0 \leq t_{N, 0} < \dots < t_{N, N} \leq h,
\end{equation*}
$X^+_{N} \coloneqq \mathbb{R}^{d (N + 1)}$ and define $R^+_N$ as, for $z\in X^+$,
\begin{equation*}
R_{N}^{+} z \coloneqq (z(t_{N, 0})^T, \dots, z(t_{N, N})^T)^T.
\end{equation*}
We define $P^+_{N}$ as the discrete Lagrange interpolation operator
\begin{equation*}
(P_{N}^{+} Z)(t) \coloneqq \sum_{n = 0}^{N} \ell_{N, n}^{+}(t) Z_{n}, \qquad t \in [0, h],
\end{equation*}
where $\{\ell^+_{N, 0}, \dots, \ell^+_{N, N}\}$ is the Lagrange basis for the chosen nodes and $Z_n\in\mathbb{R}^d$ is the $n$-th block of $d$ components of $Z$.
The space $\Pi^+_N$ is the space of polynomials of degree at most $N$ on $[0,h]$.

It is clear that \cref{H-compat-1,H-compat-2} are satisfied for all discretization indices $M$ and $N$ with $M\geq N+1$.

In \cite{BredaMasetVermiglio2012,BredaMasetVermiglio2015} the subspaces of $X$ and $X^+$ of Lipschitz continuous functions were chosen for $\widehat{X}$ and $\widehat{X}^+$, respectively, with norms defined, for $\phi\in \widehat{X}$ and $z\in \widehat{X}^+$, by $\norm{\phi}_{\widehat{X}}\coloneqq \norm{\phi}_{X}+\Lip(\phi)$ and $\norm{z}_{\widehat{X}^+}\coloneqq \norm{z}_{X^+}+\Lip(z)$, where $\Lip(\cdot)$ denotes the Lipschitz constant.
In \cite{Liessi2018}, instead, the subspaces of $X$ and $X^+$ of absolutely continuous functions%
\footnote{Recall that an absolutely continuous function has almost everywhere in its domain a derivative which is Lebesgue-integrable.}
were chosen for $\widehat{X}$ and $\widehat{X}^+$, respectively, with norms defined, for $\phi\in \widehat{X}$ and $z\in \widehat{X}^+$, by $\norm{\phi}_{\widehat{X}}\coloneqq \norm{\phi}_{L^1([-\tau,0],\mathbb{R}^d)}+\norm{\phi'}_{L^1([-\tau,0],\mathbb{R}^d)}$ and $\norm{z}_{\widehat{X}^+}\coloneqq \norm{z}_{L^1([0,h],\mathbb{R}^d)}+\norm{z'}_{L^1([0,h],\mathbb{R}^d)}$.
In the following we will reference the first case as the Lip case and the second as the AC case.

In both cases, \cref{H-Pi-V-Xhat,H-Vp-Hhatp-Xhat,H-V-Xhat-chain} are obviously satisfied.
\Cref{H-hat-norm,H-norm-Vp-z-h} are also satisfied, with $\hat{c}_1=1$ and $\hat{c}_2=1+h$ in the Lip case and with $\hat{c}_1$ given by \cite[Theorem 8.8]{Brezis2011} and $\hat{c}_2=h(1+h/2)$ in the AC case.

In the Lip case, \cref{H-conv} is satisfied if the Lebesgue constant of the chosen nodes in $[0,h]$ is $\Lambda^+_N=o(N)$.
Indeed, by well-known results in interpolation theory (see, e.g., \cite[Corollary 1.4.2 and Theorem 4.1]{Rivlin1969}), for $z\in\widehat{X}^+$ the bound
\begin{equation*}
\norm{(\mathcal{L}_{N}^{+} - I_{X^{+}})z}_{X^{+}}
\leq (1 + \Lambda^+_{N}) E_N(z)
\leq (1 + \Lambda^+_{N}) \frac{h}{2} \frac{6\Lip(z)}{N}
\end{equation*}
holds, where $E_N(\cdot)$ is the best uniform approximation error on $[0,h]$.
Now \cref{H-conv} follows from the Banach--Steinhaus uniform boundedness theorem (see, e.g., \cite[page 269]{RoydenFitzpatrick2010}).
One notable family of nodes satisfying the required condition is that of Chebyshev zeros%
\footnote{We choose here to associate to the discretization index $N$ the $N+1$ zeros of the Chebyshev polynomial of the first kind of degree $N+1$, instead of the traditional choice of the degree $N$, to maintain consistency with the definitions given in \cref{sec:numerical}.
Of course, since we eventually consider the limit as $N\to+\infty$, there is no practical difference.}
\begin{equation}\label{cheb-zeros}
\frac{h}{2}\left(1 - \cos\left(\frac{(2n + 1) \pi}{2(N+1)}\right)\right), \qquad n \in \{0,\ldots,N\}.
\end{equation}
In fact, for these nodes the Lebesgue constant $\Lambda^+_{N}$ is bounded by $\frac{2}{\pi}\log(N+1)+1$ \cite{Rivlin1974}.

In the AC case, instead, we directly require the use of Chebyshev zeros.
Indeed, by \cite[Theorem 1]{Krylov1956-en}, if $z \in \widehat{X}^{+}$, then $\mathcal{L}_{N}^{+} z \to z$ uniformly as $N \to +\infty$ and \cref{H-conv} follows, again, by the Banach--Steinhaus theorem.

For the chosen prototype form of RFDE \cref{RFDE-prototype}, in the Lip case, \cref{H-FsVp-to-Xhat,H-FsVm-to-Xhat} are satisfied if $A$ and $B_k$ are Lipschitz continuous and there exists $m_k\in L^1([-\tau_k,-\tau_{k-1}],\mathbb{R})$ such that $\norm{C_k(t_1,\theta)-C_k(t_2,\theta)}\leq m_k(\theta)\abs{t_1-t_2}$ for all $t_1,t_2\in[s,s+h]$ and almost all $\theta\in[-\tau_k,-\tau_{k-1}]$, for $k\in\{1,\dots,r\}$.
In the AC case, instead, \cref{H-FsVp-to-Xhat,H-FsVm-to-Xhat} are satisfied if $A$ and $B_k$ are absolutely continuous and the following conditions on $C_k$ (see \cite[Proposition 24]{BredaLiessi2021}) are fulfilled for $k\in\{1,\dots,r\}$:
\begin{itemize}
\item for each compact interval $J \subset \mathbb{R}$
\begin{equation*}
\esssup_{(t, \theta) \in J \times [-\tau, 0]} \abs{C_k(t, \theta)} < +\infty;
\end{equation*}
\item the directional derivative $\partial_{(1, -1)} C_k(t, \theta)$ exists for all $t \in \mathbb{R}$ and almost all $\theta \in [-\tau, 0]$;
\item there exist $\eta_k > 0$ and an essentially bounded function $m_k$ such that for all $t \in \mathbb{R}$, almost every $\theta \in [-\tau, 0]$, and all $\eta$ with $0 < \abs{\eta} < \eta_k$ (with $\theta - \eta \in [-\tau, 0]$)
\begin{equation*}
\abs{C_k(t+\eta, \theta-\eta) - C_k(t, \theta)} \leq m_k(t+\theta) \abs{\eta}.
\end{equation*}
\end{itemize}

With the choices above, all the results of \cref{sec:convergence} hold, in particular \cref{eig-convergence,convergence-theorem}.
Moreover, we can make the error estimates more specific.
In fact, from \cref{eq:estimate} we know that
\begin{equation*}
\epsilon_N \leq \text{(constant)} \cdot \sum_{j=1}^{\nu} \norm{(\mathcal{L}_{N}^{+} - I_{X^{+}}) z^{\ast}_j}_{X^{+}},
\end{equation*}
where $z^{\ast}_{j}\in\widehat{X}^+$ for $j\in\{1,\dots,\nu\}$ are the solutions of~\cref{fixed_point} associated to the elements $\phi_{j}$ of a basis of the generalized eigenspace $\mathcal{E}_{\mu}$ associated to $\mu$.
If for each $\phi\in\mathcal{E}_{\mu}$ the corresponding solution $z^*$ of \cref{fixed_point} is of class $C^{p}$, with $p \geq 1$, by well-known results in interpolation theory (see, e.g., \cite[Theorem 1.5]{Rivlin1969} and the already cited \cite[Theorem 4.1]{Rivlin1969}) the bound
\begin{equation}\label{RFDE-final-estimate}
\begin{split}
\norm{(\mathcal{L}_{N}^{+} - I_{X^{+}})z^{\ast}_{j}}_{X^{+}} &\leq (1 + \Lambda^+_{N}) E_N(z^{\ast}_{j}) \\
&\leq (1 + \Lambda^+_{N}) \frac{6^{p + 1} \mathrm{e}^{p}}{1 + p} \left(\frac{h}{2}\right)^{p} \frac{1}{N^{p}} \, \omega_p\left(\frac{h}{2(N - p)}\right)
\end{split}
\end{equation}
holds for $N>p$, where $E_N(\cdot)$ is the best uniform approximation error and $\omega_p(\cdot)$ is the modulus of continuity of $(z^{\ast}_{j})^{(p)}$ on $[0, h]$.
Since we assumed that $\Lambda^+_N=o(N)$ (as a direct assumption in the Lip case and as a consequence of the choice of Chebyshev zeros in the AC case), we conclude that $\epsilon_{N} = o(N^{1 - p})$ and the error estimate \cref{eq:final-error} becomes
\begin{equation*}
\max_{j \in \{1, \dots, \nu\}} \abs{\mu_{N, j} - \mu} = o(N^{(1 - p)/l}).
\end{equation*}

Other families of interpolation nodes may be used also in the AC case, provided that they guarantee the validity of \cref{H-conv}.
The error estimate above is conserved as long as their Lebesgue constant still satisfies $\Lambda^+_{N} = o(N)$.
Observe that both are guaranteed by zeros of other families of classic orthogonal polynomials \cite{Brutman1997}.
We chose Chebyshev nodes here since it is the typical choice of the authors and colleagues in the implementation of this method \cite{BredaMasetVermiglio2012,BredaMasetVermiglio2015}.

In general, computing the integrals in~\cref{Fs_RFDE} (see \cref{L_RFDE}) exactly may not be possible.
In that case, an approximation $\widetilde{\mathcal{F}}_{s}$ of $\mathcal{F}_{s}$ is used, leading to a further contribution to the error.

\subsection{Pseudospectral collocation of REs}
\label{sec:pscoll-re}

For $s\in\mathbb{R}$, we consider the IVP defined by \cref{RFE} for $t>s$ and $x_s = \phi \in X$.
As long as $h\coloneqq t-s \in [0, \tau]$, this corresponds to the Volterra integral equation (VIE) of the second kind
\begin{equation*}
y(h) = \int_{0}^{h} K(h, \sigma) y(\sigma) \dd \sigma + f(h)
\end{equation*}
for
\begin{equation*}
K(h, \sigma) \coloneqq C(s + h, \sigma - h)
\end{equation*}
and
$f(h) \coloneqq \int_{h - \tau}^{0} K(h, \sigma) \phi(\sigma) \dd \sigma$.
With standard regularity assumptions on the kernel $C$, the solution exists unique and bounded in~$L^{1}$ (see \cite[Theorem 2.2]{BredaLiessi2018}).
Moreover, a reasoning on the lines of Bellman's method of steps~\cite{Bellman1961,BellmanCooke1965} allows to extend well-posedness to any $h > 0$, by working successively on $[\tau,2\tau]$, $[2\tau,3\tau]$ and so on (see also \cite{BellenZennaro2003,BellmanCooke1963} for similar arguments, and \cite[section 4.1.2]{Brunner2004} for VIEs).

For $t\geq s$, the evolution operator $U(t,s)\colon X\to X$ is defined as $U(t, s) \phi = x_{t}(\cdot; s, \phi)$, where $x(t; s, \phi)$ is the solution of the IVP.

As the spaces $X^+$ and $X^{\pm}$ we consider $L^{1}([0, h], \mathbb{R}^{d})$ and $L^{1}([-\tau, h], \mathbb{R}^{d})$, respectively, with the norms $\norm{\cdot}_{X^{+}}$ and $\norm{\cdot}_{X^{\pm}}$ being the corresponding $L^{1}$ norms.

We define the operator $V \colon X \times X^{+} \to X^{\pm}$ as
\begin{equation*}
V(\phi, z)(t) \coloneqq
\begin{cases}
z(t), & \quad t \in [0, h], \\
\phi(t), & \quad t \in [-\tau, 0).
\end{cases}
\end{equation*}
Note that $V(\phi, z)$ can have a discontinuity in $0$ even when $\phi$ and $z$ are continuous, if $\phi(0) \neq z(0)$.
This important difference with respect to \cref{sec:pscoll-rfde} calls for special attention to discontinuities and to the role of $0$.
Observe that even if $V(\phi,z)(0)$ is formally given by $z(0)$ (instead of $\phi(0)$ as we would expect from the definition of the IVP), we can still say that $V(\phi,z)_0=V(\phi,z)\restriction_{[-\tau,0]}=\phi$ since elements of $X$ are equivalence classes of functions that are equal almost everywhere.
We define the operator $\mathcal{F}_{s} \colon X^{\pm} \to X^{+}$ as
\begin{equation*}
\mathcal{F}_{s} u(t) \coloneqq \int_{- \tau}^{0} C(s+t, \theta) u(t+\theta) \dd \theta, \qquad t \in [0, h].
\end{equation*}
We show below that with some assumptions on $C$ the elements of the range of $\mathcal{F}_s$ are continuous.

To be more precise on the well-posedness of the problem, according to \cite[Theorem 2.2]{BredaLiessi2018}) the validity of \cref{H-I-FsVp} is guaranteed if the interval $[0,\tau]$ can be partitioned into finitely many subintervals $J_{1}, \dots, J_{n}$ such that, for any $s\in\mathbb{R}$,
\begin{equation}\label{esssup-eq}
\esssup_{\sigma \in J_{i}} \int_{J_{i}} \abs{C(s + t, \sigma - t)} \dd t < 1, \qquad i \in \{1, \dots, n\}.
\end{equation}

As the spaces $\widetilde{X}$ and $\widetilde{X}^+$ we consider the subspaces of functions continuous from the right, in order to make point-wise evaluation meaningful.
For $\widetilde{X}$ we cannot use continuous functions, since the value of the operator $V$ can still be discontinuous at $0$ even if its arguments are continuous, as observed above.
The space $\widetilde{X}^+$, instead, can be the subspace of continuous functions.

The operators $R_M$, $R^+_N$, $P_M$ and $P^+_N$ and the spaces $\Pi_M$ and $\Pi^+_N$ are defined as in \cref{sec:pscoll-rfde} (for $h<\tau$ see \cite{BredaLiessi2018}).

It is clear that \cref{H-compat-1,H-compat-2} are satisfied for all discretization indices $M$ and $N$ with $M\geq N$, with $\mathcal{F}_s V(\phi,z) \in \widetilde{X}^+$ being satisfied with some assumptions on $C$ as shown below.

As the space $\widehat{X}^+$ we consider the space of continuous functions $C([0, h], \mathbb{R}^{d})$ with the uniform norm, while we take $\widehat{X}\coloneqq X$.
With these choices, \cref{H-Pi-V-Xhat,H-Vp-Hhatp-Xhat,H-V-Xhat-chain} are obviously satisfied and \cref{H-hat-norm,H-norm-Vp-z-h} are satisfied with $\hat{c}_1=h$ and $\hat{c}_2=1$.

By~\cite[Corollary of Theorem Ia]{ErdosTuran1937}, if the interpolation nodes in $[0,h]$ are the Chebyshev zeros \cref{cheb-zeros}, then \cref{H-conv} is satisfied.

The following assumptions on the integration kernel $C$ imply the validity of \cref{H-FsVp-to-Xhat,H-FsVm-to-Xhat}:
\begin{itemize}
\item there exists $\gamma > 0$ such that $\abs{C(t, \theta)} \leq \gamma$ for all $t \in [0, h]$ and almost all $\theta \in [-\tau, 0]$;
\item $t \mapsto C(t, \theta)$ is continuous for almost all $\theta \in [-\tau, 0]$, uniformly with respect to~$\theta$.
\end{itemize}
Indeed, let $u \in X^{\pm}\setminus\{0_{X^{\pm}}\}$, $t \in [0, h]$ and $\epsilon > 0$.
From the continuity of translation in~$L^{1}$ there exists $\delta' > 0$ such that for all $t' \in [0, h]$ if $\abs{t' - t} < \delta'$ then $\int_{-\tau}^{0} \abs{u(t' + \theta) - u(t + \theta)} \dd \theta < \frac{\epsilon}{2 \gamma}$.
From the continuity of $C$ there exists $\delta'' > 0$ such that for all $t' \in [0, h]$ and almost all $\theta \in [-\tau, 0]$ if $\abs{t' - t} < \delta''$ then $\abs{C(t', \theta) - C(t, \theta)} < \frac{\epsilon}{2 \norm{u}_{X^{\pm}}}$.
Hence, for all $t' \in [0, h]$ if $\abs{t' - t} < \delta \coloneqq \min\{\delta', \delta''\}$ then
\begin{equation*}
\begin{split}
&\abs[\Big]{\int_{-\tau}^{0} C(t', \theta) u(t' + \theta) \dd \theta - \int_{-\tau}^{0} C(t, \theta) u(t + \theta) \dd \theta} \\
&\qquad \leq \int_{-\tau}^{0} \abs{C(t', \theta)} \abs{u(t' + \theta) - u(t + \theta)} \dd \theta + \int_{-\tau}^{0} \abs{C(t', \theta) - C(t, \theta)} \abs{u(t + \theta)} \dd \theta \\
&\qquad < \gamma \frac{\epsilon}{2 \gamma} + \frac{\epsilon}{2 \norm{u}_{X^{\pm}}} \int_{-\tau}^{0} \abs{u(t + \theta)} \dd \theta \leq \epsilon.
\end{split}
\end{equation*}
Since $\mathcal{F}_{s}0_{X^{\pm}}=0_{X^{+}}$, this shows that $\mathcal{F}_{s}(X^{\pm}) \subseteq \widehat{X}^{+}$, which implies the first part of \cref{H-FsVp-to-Xhat,H-FsVm-to-Xhat}.
Boundedness follows immediately since $\norm{\mathcal{F}_s(u)}_{\widehat{X}^+}\leq\gamma\norm{u}_{X^\pm}$.

The boundedness of $C$ assumed above implies also \cref{H-I-FsVp}.
Indeed, to satisfy \cref{esssup-eq} it is enough to choose subintervals $J_i\subseteq[-\tau,0]$ of length strictly less than $1/\gamma$.

With the choices above, all the results of \cref{sec:convergence} hold, in particular \cref{eig-convergence,convergence-theorem}.
The error estimates at the end of \cref{sec:pscoll-rfde} hold also in this case, with the only difference that a further factor $h$ appears in \cref{RFDE-final-estimate}.
The comments on other families of nodes (see \cite{BredaLiessi2018,Liessi2018} for the typical implementations of the method by the authors) and on the approximation of integrals are still valid as well.

\subsection{Piecewise pseudospectral collocation of RFDEs and REs}
\label{sec:pscoll-pw}

In applications, exact periodic solutions of delay equations are in general unknown, so numerical methods are needed.
Periodic solutions are usually approximated with continuous piecewise polynomials determined by collocating a corresponding boundary value problem on the period interval \cite{Bader1985,EngelborghsLuzyaninaIntHoutRoose2001,Ando2020,AndoBreda2020b,Ando2021,AndoBreda2023a}.
It is standard to adapt the partition of the period interval to the profile of the solution, moving away from uniform (for mesh adaptation see \cite{AscherMattheijRussell1988,EngelborghsLuzyaninaIntHoutRoose2001}).

When investigating the stability of the periodic orbit via the monodromy operator, the lack of smoothness at the partition points is reflected in the coefficients of the linearized equation and in the operator, deteriorating also the convergence of the computed Floquet multipliers.
The points where the coefficients are not differentiable should be included in the collocation grid, as is observed also in \cite[section~4]{BorgioliHajduInspergerStepanMichiels2020}.
Therefore, we consider a piecewise version of the methods of \cref{sec:pscoll-rfde,sec:pscoll-re}, discretizing the monodromy operator on a grid including the adapted partition of the period interval from the given numerical periodic solution.

When piecewise polynomials are considered, two approaches to convergence are available: the spectral elements method (SEM), consisting in increasing the degree of the polynomials while keeping the number of pieces constant, and the finite elements methods (FEM), consisting in reducing the maximum length of pieces while keeping the degree constant (if the partition is uniform, this is equivalent to increasing the number of pieces).

We thus consider a partition of $[0,h]$ and define the space $\Pi^+_N$ of continuous piecewise polynomials on this partition with a certain degree.
Here the discretization index $N$ is either the degree for the SEM or it is linked to the maximum length of pieces for the FEM (if the partition is uniform, the index can be the number of pieces).
We choose a set of interpolation nodes to be used in each piece (with suitable scaling and translation).
The operators $R^+_N$ and $P^+_N$ act as described in \cref{sec:pscoll-rfde} on each piece, with $P^+_N$ defining continuous piecewise polynomials.

As in the previous sections, we consider here the case $h\geq\tau$.
The partition of $[0,h]$ is translated by $-h$ to define a partition of $[-\tau,0]$.
The space $\Pi_M$ of continuous piecewise polynomials on this new partition with a certain degree remains defined, with the discretization index $M$ being the relevant parameter according to the choice of FEM or SEM.
Again, we choose a set of interpolation nodes to be used in each piece and the operators $R_M$ and $P_M$ act similarly to $R^+_N$ and $P^+_N$.

The spaces $X$, $X^+$, $X^\pm$, $\widetilde{X}$, $\widetilde{X}^+$, $\widehat{X}$, $\widehat{X}^+$ are defined as in \cref{sec:pscoll-rfde,sec:pscoll-re}.
\Cref{H-I-FsVp,H-Xhat} are satisfied independently of the discretization method, as long as $\widehat{X}$ and $\widehat{X}^+$ have been chosen.
The same is true for \cref{H-hat-norm,H-FsVp-to-Xhat}.
Since the elements of $\Pi^+_N$ are continuous piecewise polynomials on a finite number of pieces, \cref{H-Pi-V-Xhat} is automatically verified, both in the Lip and in the AC case for RFDEs.
Part of \cref{H-compat-1} is also clearly true; the inclusions $\Pi_M\subseteq\Pi_{M+1}$ and $\Pi^+_N\subseteq\Pi^+_{N+1}$ are obvious for the SEM, where the partition is fixed; for the FEM, instead, the partitions of $[0,h]$ and $[-\tau,0]$ identified by the discretization indices $N+1$ and $M+1$ need to be include all the points of the partitions of indices $N$ and $M$, respectively.
Given how the partition of $[-\tau,0]$ is constructed from that of $[0,h]$, \cref{H-compat-2} is verified by choosing for $\Pi_M$ a degree equal or larger than the one chosen for $\Pi^+_N$.
As for \cref{H-conv}, in the SEM case, the piecewise interpolation converges for the same reasons as with one single piece.
The FEM, instead, does not require specific choices or properties of the interpolation nodes; in this case the interpolation error is proportional to a power of the maximum length of the pieces, which ensures the convergence.
Hence, with the choices made above, the results of \cref{sec:convergence} hold, in particular \cref{eig-convergence,convergence-theorem}.

\subsection{Method of weighted residuals for RFDEs}
\label{sec:weightedresiduals}

The pseudospectral collocation belongs to the class of numerical methods relying on the \emph{method of weighted residuals}. We will now describe its application to our fixed-point equation \cref{fixed_point}.

For a function $z\in X^+$ and its approximation $\tilde{z}$, we make the ansatz
\begin{equation*}
\tilde{z}(\theta) = \sum_j a_j \hat{\phi}_j(\theta),
\end{equation*}
where $a_j\in\mathbb{R}^d$ and $\{\hat{\phi}_i\}_i$ constitute a basis for a finite dimensional function space. The residual $r$ is then defined as $r = \tilde{z}-\mathcal{F}_s V(\phi,\tilde{z})$. For a given set $\{\psi_i\}_i$ of \emph{test functions}, the equations defining the discretized problem read
\begin{equation*}
\langle r, \psi_i\rangle = 0,\qquad i = 0,\dots,N,
\end{equation*}
where the inner product is defined by $\langle f, g\rangle \coloneqq\int_{0}^h f(\theta)g(\theta)\dd\theta$. In other words, the fixed-point equation \cref{discrete_FP} is obtained by defining, for $z\in \widetilde{X}^+$,
$$
R_N^+z = (\langle z,p_0\rangle,\langle z,p_1\rangle,\dots,\langle z,p_N\rangle)^T\in\mathbb{R}^{d(N+1)}\eqqcolon X_N^+,$$
which, for appropriately chosen test functions, provides $N+1$ independent linear equations for the unknowns $\phi(\theta_{N,0}),\dots,\phi(\theta_{N,N})$. The operator $P_N^+:X_N^+\to X^+$ reconstructing $\tilde{z}$ starting from the values of $R_N^+z$ is given by
\begin{equation*}
(P_N^+Z)(\theta) = \sum_{i=0}^MZ_ip_i(\theta).
\end{equation*}
$R_M$ and $P_M$ are defined similarly, but in a piecewise fashion in the case $h<\tau$, generalizing the construction in \cref{sec:pscoll-rfde}.

When the functions $\{\hat{\phi}_i\}_{i\geq 0}$ are chosen as the Lagrange polynomials defined by Chebyshev nodes, and the test functions are the Dirac-delta distributions
\begin{equation*}
\psi_i(\theta)=\delta(\theta-\theta_i),
\end{equation*}
we retrieve the pseudospectral collocation method. For RFDEs, another method using the same set of basis functions but the Legendre polynomials $\{p_i\}_{i\geq 0}$ on $[0,h]$ as test functions is the \emph{spectral element method}, first introduced in \cite{KhasawnehMann2011a} for time-periodic linear RFDEs with discrete delays, and then extended to wider classes of RFDEs in \cite{KhasawnehMann2011b,KhasawnehMann2013,LehotzkyInspergerStepan2016}, albeit without a rigorous convergence proof.

For this case, we consider again the choices in \cref{sec:pscoll-rfde} for the spaces $X$, $X^+$, $X^\pm$, $\widetilde{X}$ and $\widetilde{X}^+$.
Due to the properties of the restriction and prolongation operators, \cref{H-compat-1,H-compat-2} hold again for all $M,\,N$ such that $M>N$. Finally, observe that, if $f$ is a Lipschitz continuous function on $[0,h]$ and $S_N$ is the sum of the first $N+1$ elements of its Legendre series, we have
\begin{equation*}
\norm{f-S_N}_{X^+} = O(\log(N)/N)
\end{equation*}
and, for any $[a,b]\subset(0,h)$,
\begin{equation*}
\norm{(f-S_N)\restriction_{[a,b]}}_{X^+} = O(N^{-1/2}),
\end{equation*}
\cite[pp. 31--32]{Jackson1930}. Thus, \cref{H-Xhatp} is satisfied by choosing again $\widehat{X}^+$ as the subspace of Lipschitz continuous functions of $X^+$.
Observe that \cref{H-I-FsVp} is satisfied for the reasons given in \cref{sec:pscoll-rfde}, since it only depends on the original problem and not on the method; the same holds for \cref{H-Xhat} for the given $\widehat{X}^+$ if Lipschitz continuous functions are chosen also for $\widehat{X}$, as in \cref{sec:pscoll-rfde}.

\section{Concluding remarks}
\label{sec:discussion}

In this work, we have seen how the convergence of various methods for the approximation of evolution operators of delay equations can be analyzed in a unified framework, which can in turn be used to provide formal proofs also for other numerical techniques.
For example, besides the methods already mentioned above, an interesting application would be a Fourier-based reduction, also aiming at the approximation of Lyapunov exponents (see, e.g., \cite{BredaVanVleck2014,BredaLiessi2025}).

The authors and colleagues are also working on the approximation of evolution operators obtained from the linearization of RFDEs with state-dependent delays: the proof of convergence will be provided by means of the proposed framework.

Moreover, we are interested in the stability analysis of structured population models formulated as PDEs (see, e.g., \cite{ScarabelBredaDiekmannGyllenbergVermiglio2021,AndoDeReggiLiessiScarabel2023}).
A possible future development could be to draw inspiration from this work and construct a framework for the convergence proof in that case.

\section*{Acknowledgements}

A.\,A., D.\,B.\ and D.\,L.\ are members of INdAM Research group GNCS and of UMI research group ``Mo\-del\-li\-sti\-ca socio-epidemiologica''. The research collaboration was fostered by the workshop \emph{Towards rigorous results in state-dependent delay equations} held at the Lorentz Center in Leiden on 4--8 March 2024.
The work was partially supported by the Italian Ministry of University and Research (MUR) through the PRIN 2022 project (No.\ 20229P2HEA) ``Stochastic numerical modelling for sustainable innovation'', Unit of Udine (CUP G53C24000710006).


\appendix

\section{The case \texorpdfstring{$h<\tau$}{h<τ}}
\label{sec:h-lt-tau}

Let us assume now that $h<\tau$.
As anticipated in \cref{sec:numerical}, in this case for the discretization of $X$ we adopt a piecewise approach.
More precisely, let $Q$ be the minimum positive integer $q$ such that $q h \geq \tau$.
Observe that $Q > 1$.
Let also $\theta^{(q)} \coloneqq - q h$ for $q \in \{0, \dots, Q - 1\}$ and $\theta^{(Q)} \coloneqq - \tau$.
On each interval $[\theta^{(q)}, \theta^{(q - 1)}]$ for $q\in\{1,\dots,Q\}$ we consider a reduction method and we let the operators $R_M$, $P_M$ and $\mathcal{L}_M$ act separately interval-wise.
The generalized polynomials in $\Pi_M$ are then defined piecewise and in each interval $[\theta^{(q)}, \theta^{(q-1)}]$ we consider a basis $\{\phi^{(M,q)}_0,\dots,\phi^{(M,q)}_M\}$.
In principle the methods can be different on each interval: again, compatibility conditions are required (\cref{H-compat-2-b} below).

We now present alternative versions of some of the hypotheses of \cref{sec:hypotheses} tailored to this case.

\begin{H-hyp}[label={},leftmargin=0pt]
\addtocounter{H-hypi}{2}
\item
\begin{H-hyp}[label=(H\arabic{H-hypi}.\arabic*b),leftmargin=*]
\addtocounter{H-hypii}{3}
\item\label{H-V-Xhat-chain-b} Define for each $q\in\{1,\dots,Q\}$ the space
\begin{equation*}
\widehat{X}^{(q)}\coloneqq \{\phi\restriction_{[\theta^{(q)}, \theta^{(q-1)}]} \mid \phi\in \widehat{X}\}
\end{equation*}
and assume that
\begin{itemize}
\item $V(\phi,z)_h\restriction_{[-h,0]}\in \widehat{X}^{(1)}$ for each $(\phi,z)\in X\times\widehat{X}^+$,
\item if $\phi\in X$ and $\phi\restriction_{[\theta^{(q)},\theta^{(q-1)}]} \in \widehat{X}^{(q)}$ for each $q\in\{1,\dots,Q\}$ then $\phi\in\widehat{X}$.
\end{itemize}
\end{H-hyp}
\end{H-hyp}
\begin{H-hyp}[label=(H\arabic*b),leftmargin=*]
\addtocounter{H-hypi}{4}
\item\label{H-compat-2-b} Define for each $q\in\{1,\dots,Q\}$ the spaces
\begin{equation*}
X^{(q)}\coloneqq \{\phi\restriction_{[\theta^{(q)}, \theta^{(q-1)}]} \mid \phi\in X\},
\qquad
\Pi^{(q)}_M \coloneqq \spanop\{\phi^{(M,q)}_0,\dots,\phi^{(M,q)}_M\}.
\end{equation*}
Define also the shift operators $S_1\colon X^+\to X^{(1)}$ and, for each $q\in\{2,\dots,Q\}$, $S_q\colon X^{(q-1)}\to X^{(q)}$: their action is translation of the variable backward by $h$, e.g., $S_q(\phi)(\theta)=\phi(h+\theta)$; in the case of $q=Q$ the functions are also restricted to the interval $[\theta^{(Q)}, \theta^{(Q-1)}]=[-\tau, -(Q-1)h]$ after the shift.
Assume, for each $q\in\{2,\dots,Q\}$, that $S_q(\Pi^{(q-1)}_M)\subseteq\Pi^{(q)}_M$ and, for each $(\phi,z)\in X\times\Pi^+_N$, that $V(\phi,z)_h\restriction_{[\theta^{(1)}, \theta^{(0)}]}\in\Pi_M^{(1)}$.
\end{H-hyp}

Using these hypotheses, the results of \cref{sec:convergence} can be proved also in the case $h<\tau$.
The relevant proofs follow; the corresponding statements are unchanged except for the fact that \cref{H-V-Xhat-chain-b,H-compat-2-b} are used instead of \cref{H-V-Xhat-chain,H-compat-2}.

\begin{proof}[Proof of \cref{hTMN-hTN-same} with $h<\tau$]
Let $\phi\in\widehat{X}$ and let $w^{\ast} \in X^{+}$ be the unique solution of~\cref{collocation_hat}.
Observe that $w^*\in\Pi^+_N$.
Thanks to \cref{H-compat-2-b}, we have
$V(\phi,w^*)_h\restriction_{[\theta^{(1)}, \theta^{(0)}]} \in\Pi_M^{(1)}$.
Moreover, for $q\in\{2,\dots,Q\}$ and $\theta\in[\theta^{(q)}, \theta^{(q-1)}]$
\begin{equation*}
V(\phi,w^*)_h\restriction_{[\theta^{(q)}, \theta^{(q-1)}]}(\theta)
=V(\phi,w^*)(h+\theta)
=\phi(h+\theta)
=S_q(\phi\restriction_{[\theta^{(q-1)}, \theta^{(q-2)}]})(\theta),
\end{equation*}
since $h+\theta\in[\theta^{(q-1)}, \theta^{(q-2)}]\subseteq[-\tau,0]$.
By induction on $q\in\{2,\dots,Q\}$, since at each step $\phi\restriction_{[\theta^{(q-1)}, \theta^{(q-2)}]}\in\Pi_M^{(q-1)}$,
\begin{equation*}
V(\phi,w^*)_h\restriction_{[\theta^{(q)}, \theta^{(q-1)}]}=S_q(\phi\restriction_{[\theta^{(q-1)}, \theta^{(q-2)}]})\in\Pi_M^{(q)},
\end{equation*}
again thanks to \cref{H-compat-2-b}.
Hence $\widehat{T}_{N} \phi = V(\phi,w^*)_h\in\Pi_M$.
The proof then proceeds like in \cref{hTMN-hTN-same}.
\end{proof}

\begin{proof}[Proof of \cref{T-hTN-restriction-to-Xhat} with $h<\tau$]
The proof changes only after \cref{eq:phi-psi}.
From \cref{H-V-Xhat-chain-b} we get $V(\phi,z^*)_h\restriction_{[-h,0]}\in\widehat{X}^{(1)}$, hence, from \cref{eq:phi-psi}, $\phi\restriction_{[\theta^{(1)},\theta^{(0)}]}\in\widehat{X}^{(1)}$.
Observe that, for $\theta\in[-\tau,-h]$,
\begin{equation*}
V(\phi,z^*)_h\restriction_{[-\tau,-h]}(\theta)
=V(\phi,z^*)(h+\theta)
=\phi(h+\theta),
\end{equation*}
since $h+\theta\in[-\tau+h,0]$.
Hence, for each $q\in\{2,\dots,Q\}$, from \cref{eq:phi-psi} we get
\begin{equation*}
\phi\restriction_{[\theta^{(q)},\theta^{(q-1)}]}
=\lambda^{-1}(\psi+V(\phi,z^*)_h)\restriction_{[\theta^{(q)},\theta^{(q-1)}]}
=\lambda^{-1}(\psi\restriction_{[\theta^{(q)},\theta^{(q-1)}]}+\phi(h+\cdot)\restriction_{[\theta^{(q)},\theta^{(q-1)}]}).
\end{equation*}
Since $\phi(h+\cdot)\restriction_{[\theta^{(q)},\theta^{(q-1)}]}$ is $\phi\restriction_{[\theta^{(q-1)},\theta^{(q-2)}]}$ with the variable translated backward by $h$, by induction on $q$ we can conclude that $\phi\restriction_{[\theta^{(q)},\theta^{(q-1)}]}\in\widehat{X}^{(q)}$ and thus, from \cref{H-V-Xhat-chain-b}, that $\phi\in\widehat{X}$.
The proof concludes as in \cref{T-hTN-restriction-to-Xhat}.
\end{proof}

\end{document}